\documentclass[12pt,english]{amsart}
\usepackage{geometry}                
\usepackage{graphicx,graphics,leftindex}
\usepackage{amssymb}
\usepackage{hyperref}
\usepackage{epstopdf}
\usepackage{tikz}
\usepackage{tikz-cd}
\theoremstyle{plain}
\newtheorem{thm}{Theorem}[section]
\newtheorem{prop}[thm]{Proposition}
\newtheorem{lem}[thm]{Lemma}
\newtheorem{conj}[thm]{Conjecture}
\newtheorem{cor}[thm]{Corollary}
\newtheorem{algo}[thm]{Algorithm}
\theoremstyle{definition}
\newtheorem{defn}[thm]{Definition}

\theoremstyle{remark}
\newtheorem{remark}[thm]{Remark}

\newcommand{\li}{\leftindex _}
\newcommand{\lu}{\leftindex[I]^}
\newcommand{\di}{\underline{\hspace{2mm}}}

\usepackage{float}
\usepackage{amscd}

\title[Special Matchings, Brenti's Conjecture, and the CIC]{Special Matchings, Brenti's Conjecture, and the Combinatorial Invariance Conjecture}
\author[Caselli, Marietti]{Fabrizio Caselli, Mario Marietti}

\date{}

\address{Fabrizio Caselli, Dipartimento di Matematica, Universit\`a degli Studi di Bologna,
Piazza di Porta San Donato 5, 40126 Bologna, Italy}
\address{Mario Marietti, Dipartimento  di Ingegneria Industriale e Scienze Matematiche, Universit\`a Politecnica delle Marche, via Brecce Bianche, 60131 Ancona,  Italy}

\email{fabrizio.caselli@unibo.it}
\email{m.marietti@univpm.it}

\subjclass[2020]
{ 05E10 - 05E16 (primary), 20F55  (secondary)}
\keywords{Bruhat order, Coxeter groups, Symmetric groups, Special matchings, Kazhdan--Lusztig polynomials, Combinatorial invariance}

\begin{document}

\maketitle

\begin{abstract}
In this work, we settle a problem that dates back to the early 2000s. We provide a complete characterization of special matchings of arbitrary Bruhat intervals in Coxeter groups of type~$A$ and apply this result to prove a conjecture of Brenti from 2003 concerning the computation of Kazhdan--Lusztig $R$-polynomials via special matchings. This yields new evidence in support of the Combinatorial Invariance Conjecture.
\end{abstract}

\section{Introduction}

The Combinatorial Invariance Conjecture, independently formulated by Lusztig and Dyer in the 1980s, is one of the central open problems in the theory of Kazhdan--Lusztig polynomials. The conjecture predicts that Kazhdan--Lusztig polynomials are entirely determined by the poset structure of Bruhat intervals.

\begin{conj}
\label{comb-inv-con}
The Kazhdan--Lusztig polynomial $P_{u,v}(q)$ depends only on the isomorphism type of the interval $[u,v]$ as a partially ordered set under Bruhat order.
\end{conj}

\noindent
This conjecture is equivalent to the same statement for the Kazhdan--Lusztig $R$-polynomials.

Despite decades of research, the Combinatorial Invariance Conjecture remains widely open. It is known in several remarkable special cases, including intervals of rank at most $6$ \cite{BGL}, intervals of rank at most $8$ in Weyl groups and at most $10$ in type $A$ \cite{EM,EMS}, intervals that are lattices \cite[7.23]{Dyeth}, \cite{Bre94}, intervals in type $\tilde A_2$ \cite{BLP}, and elementary intervals \cite{BG}. Moreover, a parabolic analogue of the conjecture is known to hold for intervals starting from the identity \cite{Mtrans,M}. At the same time, a natural poset-theoretical strengthening of the conjecture is known to fail \cite{BCM2,MJaco}, highlighting the subtle nature of the problem.

A major challenge in attacking the Combinatorial Invariance Conjecture is to find a purely combinatorial way to compute Kazhdan--Lusztig polynomials from the isomorphism class of the underlying interval alone.
Recently, a new combinatorial object, namely that of a hypercube decomposition, was introduced and studied in \cite{BBDVW,DVBBZTTBBJLWHK} as a possible approach to the Combinatorial Invariance Conjecture. This notion has already generated several further developments (see \cite{BG,B-M,EM2,G-W}).

Another less recent attempt in this direction is the theory of special matchings, introduced by Brenti in \cite{Bre04}. Special matchings are purely combinatorial maps on Bruhat intervals that mimic multiplication by simple reflections. Brenti proved that, for intervals starting from the identity element in Coxeter groups of type $A$, special matchings completely recover the recursive structure of $R$-polynomials. 
Indeed, special matchings  can replace multiplication by Coxeter generators in the recursive formula for $R$-polynomials, yielding in particular a constructive proof of combinatorial invariance for intervals starting from the identity element in Coxeter groups of type $A$. 

Motivated by these results (generalized to arbitrary Coxeter groups in \cite{BCM1}), Brenti conjectured in 2003 (see \cite{Bre03}) that the same phenomenon should hold for arbitrary Bruhat intervals in type $A$.

\begin{conj}[Brenti]
Let $W$ be a Coxeter group of type $A$. Let $u,v \in W$, and let $M$ be a special matching of the Bruhat interval $[u,v]$. Then, for all $x,y \in [u,v]$ with $x\le y$, we have
\[
R_{x,y}(q)=
\begin{cases}
    R_{M(x),M(y)}(q)
    & \textrm{if } M(x)\lhd x,\ M(y)\lhd y,\\
    R_{M(x),M(y)}(q)
    & \textrm{if } M(x)\rhd x,\ M(y)\rhd y,\\
    (q-1)R_{x,M(y)}(q)+qR_{M(x),M(y)}(q)
    & \textrm{if } M(x)\rhd x,\ M(y)\lhd y,\\
    q^{-1}R_{M(x),M(y)}(q)+(q^{-1}-1)R_{M(x),y}(q)
    & \textrm{if } M(x)\lhd x,\ M(y)\rhd y.
\end{cases}
\]
\end{conj}

Although special matchings have been the subject of several studies and are well understood for Bruhat intervals starting from the identity \cite{BCM1,CM1,CM2}, their structure on arbitrary Bruhat intervals has remained largely mysterious, and Brenti's conjecture has remained open for more than twenty years.

The main achievement of this paper is a complete classification of special matchings of arbitrary Bruhat intervals in Coxeter groups of type~$A$, and, as a consequence, we prove Brenti's 2003 conjecture. 
More precisely, we show that every special matching arises from a remarkably rigid combinatorial mechanism involving parabolic decompositions.
 In particular, for each special matching $M$ of an arbitrary Bruhat interval $[u,v]$ of a Coxeter system $(W,S)$ of type $A$, there exist a subset $J$ of $S$ and an element $s$ of $J$ such that, for all $x\in [u,v]$, we have
$$M(x)= \li J x s \lu J x,$$
where $x=\li J x  \lu J x$ is the standard left parabolic decomposition. 
An analogous statement holds for the right parabolic decomposition.

Besides proving Brenti's conjecture, our results provide further evidence that special matchings capture fundamental combinatorial features of Bruhat order and constitute a promising framework for approaching the Combinatorial Invariance Conjecture.
Indeed, we state a conjectural special matchings-based algorithm to compute the Kazhdan--Lusztig $R$-polynomial of any type $A$ Bruhat interval starting only from the knowledge of its isomorphism class.

Our results show that this conjectural algorithm is, in fact, equivalent to the Combinatorial Invariance Conjecture for Coxeter groups of type~A.

The paper is organized as follows. 
Section~\ref{preli} reviews the background material. Section~\ref{tuttiquanti} provides some preliminary results that hold for arbitrary Coxeter groups. From Section~\ref{cinque} on, we restrict our attention to the Coxeter groups of type $A$. In Section~\ref{cinque}, we give a characterization of all $J$ and $s$ such that the map $x \mapsto \li J x s \lu J x$ restricts to a special matching of the interval.
In Section~\ref{struttura}, we present several structural properties of special matchings.
In Section~\ref{badcasesI}, we suppose that there is an interval and a special matching of it that do not fall under the desired characterization and we prove several results that will lead to a contradiction in Section~\ref{badcases}.
Finally, in Section~\ref{mainsection}, we present the main results of this work.

\section{Preliminaries}
\label{preli}
In this section, we  review some background material. 

For $n,m\in \mathbb N^+$ with $n\le m$, we let
$[n]=\{1,2,\ldots,n\}$,
$[n,m]=\{n,n+1,\ldots,m\}$,
and $(n,m)=[n,m]\setminus\{n,m\}$.

We fix our notation on a Coxeter system $(W,S)$ in the following list:
\smallskip 

$
\begin{array}{@{\hskip-1.3pt}l@{\qquad}l}
e &  \textrm{the identity of $W$}, 
\\
\ell  &  \textrm{the length function of $W$ with respect to $S$},
\\
w_0 &  \textrm{the longest   element of $W$ (when $W$ is finite)},
\\
T &  \textrm{the set of  reflections of $W$, i.e. } \{ w s w ^{-1} : w \in W, \; s \in S \}, 
\\
\leq & \textrm{the Bruhat order on $W$ (as well as usual order on $\mathbb R$)},
\\
\textrm{$[u,v]$} &  \textrm{the (Bruhat) interval generated by $u,v\in W$, i.e. } \\
& \{ w \in W \, : \; u \leq w \leq v \},
\\
B(W)  &  \textrm{the  Bruhat graph  of $W$},
\\
\ell(u,v) & \ell(v)-\ell(u)
\\
D_L(w) &  \textrm{the left descent set of $w$, i.e. } \{ s \in S : \; \ell( sw) < \ell(w ) \}, 
\\
D_R(w) & \textrm{the right descent set of $w$, i.e. } \{ s \in S : \; \ell(w  s) < \ell(w ) \}, 
\\
W_J & \textrm{the parabolic subgroup generated by a subset $J$ of $S$},
\\
\lu J W & \textrm{the set of left minimal $W_J$-coset representatives, i.e.} \\
& \{w\in W: D_L(w)\cap J = \emptyset\},
\\
W^J & \textrm{the set of right minimal $W_J$-coset representatives, i.e.} \\
& \{w\in W: D_R(w)\cap J = \emptyset\},
\\
w=\li J w \lu Jw & \textrm{the decomposition of $w$ with $\li J w \in W_J$ and $\lu Jw \in \lu J W$},
\\
w=w_J \, w^J & \textrm{the decomposition of $w$ with $w_J \in W_J$ and $w^J \in W^J$},
\\
x\lhd y  &  \textrm{$x$ is covered by $y$ or $y$ covers $x$, i.e. $x<y$ and $[x,y]=\{x,y\}$}.
\end{array}$
\bigskip

The \emph{Bruhat graph} of $W$, denoted $B(W)$, is the edge-labelled directed graph whose vertex set is $W$ and where  $u \xrightarrow{t} v$ if and only if $\ell(u)<\ell(v)$ and $vu^{-1}=t \in T$. The \emph{Bruhat order} is the partial order on $W$ where $u \leq v$ whenever there is a (directed) path from $u$ to $v$ in $B(W)$ (see, for example,  \cite[\S 2.1]{BB}).
There is a well-known characterization of the Bruhat order (see \cite[\S 2.2]{BB}), which we use in this work without explicit mention.
By a {\em subword} of a word $s_{1}s_{2} \cdots  s_{q}$  we mean
a word of the form
$s_{i_{1}} s_{i_{2}} \cdots  s_{i_{k}}$, where $1 \leq i_{1}< \cdots
< i_{k} \leq q$.
Given $w$ in $W$,   a {\em reduced expression} of $w$ is a word $s_1s_2\cdots s_q$ such that $w=s_1s_2\cdots s_q$ and $\ell(w)=q$. 
\begin{thm}[Subword Property]
\label{subword}
Let $u,v \in W$. The following are equivalent:
\begin{itemize}
\item $u \leq v$ in the Bruhat order,
\item  every reduced expression of $v$ has a subword that is 
a reduced expression of $u$,
\item there exists a  reduced expression of $v$ having a subword that is 
a reduced expression of $u$.
\end{itemize}
\end{thm}

The Bruhat order provides $W$ with the structure of a graded poset having $\ell$ as its rank function.

Let $u,v \in W$. Given $w\in[u,v]$, we say that $w$ is an atom (respectively, a coatom) of $[u,v]$ provided that $ u\lhd w$ (respectively, $w\lhd v$).
The {\em Hasse diagram} of $[u,v]$ is the directed graph having the elements of $[u,v]$ as vertices and $ \{  x \rightarrow y :  x \lhd y\}$ as its set of directed edges. When we draw a Hasse diagram, we draw edges instead of directed edges implying that the edges point upward. 

We say that a poset is a $k$-crown if its Hasse diagram is isomorphic to the graph depicted in Figure~\ref{esporoma} (the graph has $2k+2$ vertices).  

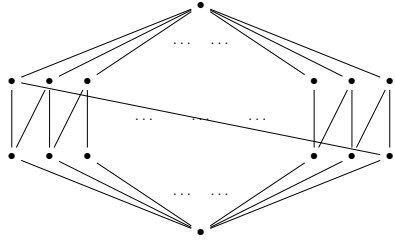
\begin{figure}[h]
    \centering
\scalebox{.5}{
    \begin{tikzpicture}
    \node (u) at (0,0) {$\bullet$};
    
    \node (a1) at (-5,2) {$\bullet$};
    \node (a2) at (-4,2) {$\bullet$};
    \node (a3) at (-3,2) {$\bullet$};
    \node (b1) at (-5,4) {$\bullet$};
    \node (b2) at (-4,4) {$\bullet$};
    \node (b3) at (-3,4) {$\bullet$};
    
	\node (a6) at (5,2) {$\bullet$};
    \node (a5) at (4,2) {$\bullet$};
    \node (a4) at (3,2) {$\bullet$};
    \node (b6) at (5,4) {$\bullet$};
    \node (b5) at (4,4) {$\bullet$};
    \node (b4) at (3,4) {$\bullet$};    
    
    \node (v) at (0,6) {$\bullet$};
    
    \node at (-0.5,5) {$\ldots$};
    \node at (0.5,5) {$\ldots$};
    
    \node at (-1.5,3) {$\ldots$};
    \node at (0,3) {$\ldots$};
    \node at (1.5,3) {$\ldots$};

    \node at (-0.5,1) {$\ldots$};
    \node at (0.5,1) {$\ldots$};
    
    \path[-]

    (u) edge (a1)
    (u) edge (a2)
    (u) edge (a3)
    
    (u) edge (a4)
    (u) edge (a5)
    (u) edge (a6)
    
    (a1) edge (b1)
    (a1) edge (b2)
    (a2) edge (b2)
    (a2) edge (b3)
    (a3) edge (b3)
    
    (a4) edge (b4)
    (a4) edge (b5)
    (a5) edge (b5)
    (a5) edge (b6)
    (a6) edge (b6)
    
    (a6) edge (b1)
    
    (b1) edge (v)
    (b2) edge (v)
    (b3) edge (v)
    
    (b4) edge (v)
    (b5) edge (v)
    (b6) edge (v);    
\end{tikzpicture}
 }
    \caption{A $k$-crown}
    \label{esporoma}
\end{figure}

\begin{lem}
\label{k_corone}
Let $W$ be an arbitrary Coxeter group. Let $u,v \in W$, with $u \leq v$ and $\ell(v) - \ell(u) =3$. Then $[u,v]$ is a $k$-crown for some $k \geq 2$. If furthermore $W$ is a Coxeter group of type $A$, then $[u,v]$ is a $k$-crown for some $k \in \{2,3,4\}$.
\end{lem}

The following are two useful ways to decompose elements  of $W$ (see \cite[\S 2.4]{BB}).
\begin{prop}
\label{fattorizzo}
Let  $J \subseteq S$ and $w\in W$. 
\begin{enumerate}
\item[(i)] 
There is a unique factorization $w=w^{J}  w_{J}$ 
with $w^{J} \in W^{J}$ and $w_{J} \in W_{J}$; this factorization satisfies $\ell(w)=\ell(w^J)+\ell(w_J)$.
\item[(ii)] There is a unique factorization $
w={_{J} w} \, {^{J}\! w}$, where ${_{J} w} \in W_{J}$, ${^{J} \!
w} \in \,  ^{J} W$; this factorization satisfies $\ell(w)=\ell({_{J} w}) + \ell( {^{J}\! w})$.
\end{enumerate}
\end{prop}

These decompositions satisfy the following property.
\begin{prop}
\label{projection}
Let  $J \subseteq S$. Let $u,v \in W$, with $u\leq v$. Then $u^J\leq v^J$ and $\lu J  u \leq \lu J v$.
\end{prop}


Throughout this work, whenever $(W,S)$ is the Coxeter system of type $A_{n-1}$, we identify $W$ with the symmetric group on $[n]$ and 
$S$ with $\{s_i=(i,i+1): i\in [n-1]\}$,
that is, we consider the standard Coxeter realization of the symmetric group.
When working with permutations, we use both cycle notation and one-line notation interchangeably. Given a permutation $z\in W$, its one-line notation is
\[
z=[z_1,\ldots,z_n],
\]
where $z_k=z(k)$ for all $k\in [n]$. For convenience, we freely identify a permutation with its one-line notation.

The following results hold for Coxeter groups of type $A$.
\begin{lem}
    \label{quadrati}
    Let $W$ be a Coxeter group of type $A$. Let $u,v\in W$, with $u < v$. Let $u \stackrel{t_1}{\rightarrow} x_1\stackrel{r_1}{\rightarrow} v$ and $u \stackrel{t_2}{\rightarrow} x_2\stackrel{r_2}{\rightarrow} v$ be the two paths from $u$ to $v$. Then the following hold.
    \begin{enumerate}
        \item If $t_1 r_1 = r_1 t_1$, then $t_1=r_2$ and $t_2=r_1$.
        \item If $t_1 r_1 \neq r_1 t_1$, then $|\{t_1,r_1\}\cap \{t_2,r_2\}|=1$. Moreover, suppose without loss of generality $t_1=r_2=(a,b)$ and $r_1=(b,c)$, then $t_2=(a,c)$ and either $c<\min(a,b)$ or $c>\max(a,b)$. 
    \end{enumerate}
\end{lem}

\begin{thm} [Tableau Criterion] Let $W$ be a Coxeter group of type $A$. Let $u,v \in W$. Let $u_{i,k}$ be the $i$-th smallest 
element in the increasing rearrangement of $u(1),u(2), \ldots, u(k)$, and similarly
define $v_{i,k}$. Then $u \leq v$ if and only if $u_{i,k} \leq v_{i,k}$ for all $k\in [n]$ and $i \in [k]$.
\end{thm}
We define the tableau of $u$ by

\[
T(u)=\begin{tabular}{|cccccc|}
\hline
$u_{1,1}$&&&&& \\
\hline
$u_{2,1}$& $u_{2,2}$&&&&\\
\hline 
$u_{3,1}$& $u_{3,2}$& $u_{3,3}$&&&\\
\hline
$\ldots$&$\ldots$&$\ldots$&$\ddots$&&\\
\hline
$u_{n,1}$&$u_{n,2}$&&&&$u_{n,n}$\\
\hline

\end{tabular}
\]
and hence $u\leq v$ if and only if $T(u)\leq T(v)$ entrywise. In our discussion, we will often have to deal with permutations that agree on a given subset $R$ of $[n]$. Whenever this happens, we use the \emph{restricted} one-line notation of a permutation $u$, i.e. we write only the entries $u_k$ for $k\notin R$, separated by a symbol $\di$.
For example, we write $u=[\di a \di b \di c \di]$ and $v=[\di c \di a \di b \di]$ to mean that $u$ and $v$ agree in all positions that are not occupied by $a,b,c$, and that these entries appear in the given order.
Moreover, whenever $u$ and $v$ agree on a subset of $R$, we can consider the restricted tableaux of $u$ and $v$, which is defined as the ordinary tableau but using only the displayed coefficients. For example, if $u$ and $v$ are as above and $a<b<c$, then the restricted tableaux of $u$ and $v$ are  $\begin{tabular}{|ccc|}\hline $a$&&\\\hline $a$&$b$&\\ \hline $a$&$b$&$c$\\ \hline \end{tabular}$ and $\begin{tabular}{|ccc|}\hline $c$&&\\\hline $a$&$c$&\\ \hline $a$&$b$&$c$\\ \hline \end{tabular}$ respectively.
It is straightforward to verify that restricted tableaux behave naturally with respect to the Bruhat order: $u<v$ if and only if the restricted tableau of $u$ is entrywise smaller than the restricted tableau of $v$.

The following is a simple but very useful observation.
\begin{remark}Let $W$ be a Coxeter group of type $A$, $x,y,z,v\in W$, with $x,y\leq v$. If $T(z)\leq \max(T(x),T(y))$ (where the maximum of two tableaux is defined coefficientwise), then $z\leq v$.
\end{remark}

Let $W$ be a Coxeter group, and $u,v \in W$. A {\em matching} of the Bruhat interval $[u,v]$ is an involution
$ M : [u,v]\rightarrow [u,v] $ such that $\{w,M(w)\}$ is an edge in the Hasse diagram of $[u,v]$, for all $w\in [u,v]$.
A matching \( M \) of \( [u,v] \) is {\em special} if\[
x\lhd y\Longrightarrow M(x)\leq M(y),\]
 for all \( x,y\in [u,v] \) such that \( M(x)\neq y \). 

The following result (see \cite[Lemma~2.6]{MJaco}) is a generalization  of the well-known Lifting Property (see \cite[Proposition~2.2.7]{BB}).
\begin{lem}
\label{lpfsm}
 (Lifting Property for special matchings) Let $W$ be a Coxeter group, and $u,v \in W$. Let $M$ be a special matching of $[u,v]$, and let $x, y \in [u,v]$, with $x \leq  y$. Then
\begin{enumerate}
\item if $M(y) \lhd y$ and $M(x) \lhd x$ then $M(x) \leq M(y)$,
\item if  $M(y) \rhd y$ and $M(x) \rhd x$ then $M(x) \leq M(y)$,
\item if $M(y) \lhd y$ and $M(x) \rhd x$ then $M(x) \leq y$ and $x \leq M(y)$.
\end{enumerate}
\end{lem}

\begin{lem}
\label{si_restringe}
Let $W$ be a Coxeter group, and $u,v \in W$.  Let $M$ be a special matching of $[u,v]$. Let $x,y \in [u,v]$ be such that $x<y$, $M(x)\rhd x$ and $M(y)\lhd y$. Then $M$ restricts to a special matching of $[x,y]$.
    
\end{lem}

Let $W$ be a Coxeter group, and $u,v \in W$.
Given $ s\in D_{L}(v)\setminus D_L(u) $, we denote by  
$\lambda_{s}$ the special matching  of $[u,v]$ 
given by $\lambda_{s}(z)=sz $ for all $ z $ in $[u,v]$, and, given $ s\in D_{R}(v)\setminus D_R(u) $, we denote by  
$\rho_{s}$ the special matching  of $[u,v]$ 
given by $\rho_{s}(z)=zs $ for all $ z $ in $[u,v]$. That these are indeed special matchings follows from the classical Lifting Property for Coxeter groups.

Special matchings of intervals starting from the identity element in arbitrary Coxeter groups have been studied in \cite{BCM1,CM1,CM2}.
In this work, we study special matchings of arbitrary intervals in Coxeter groups of type $A$.

\begin{remark}
In this work, we use the color red in two different ways: in figures and in one-line notations of permutations.
In figures, red edges indicate the edges belonging to the matching under consideration. These edges are also drawn thicker than the others.
In one-line notations, some entries may be written in red to indicate that they are \lq\lq red entries\rq\rq, as defined in Section~\ref{cinque} (note that entries written in black are not necessarily \lq\lq black entries\rq\rq).
In both cases, the use of the color red is not essential for understanding the arguments, but is intended only as a visual aid for the reader.
\end{remark}

\section{Preliminary results for arbitrary Coxeter groups}
\label{tuttiquanti}
In this section, we present some results that hold for all Coxeter groups.
Throughout this section, unless otherwise stated, $(W,S)$ denotes an arbitrary Coxeter system.

\begin{defn}
    Let $J\subseteq S$ and $s\in J$. We let $\lambda_s^J$ be the matching of $W$ given by
\[
\lambda_s^J(w)= \li J w \, s \, \lu J w,
\]
for all $w$ in $W$.
Symmetrically, we let $\rho_s^J$ be the matching of $W$ given by
\[
\rho_s^J(w)=    w^J \, s \,  w_J,
\]
for all $w$ in $W$.
\end{defn}

Note that, given $s \in S$, the three matchings  $\lambda_s^{\{s\}}$, $\rho_s^S$, and $\lambda_s$ coincide, as well as the matchings $\rho_s^{\{s\}}$, $\lambda_s^S$, and $\rho_s$.

Let $u,v \in W$, with $u\leq v$. The map $x\mapsto x^{-1}$ gives a poset isomorphism between $[u,v]$ and $[u^{-1},v^{-1}]$. Hence, we have a bijection between matchings of $[u,v]$ and matchings of $[u^{-1},v^{-1}]$, sending a matching $M$ to the matching $\overline{M}$ defined by $\overline{M}(x^{-1})= (M(x))^{-1}$. This bijection sends $\lambda_s^J$ to $\rho_s^J$ and $\rho_s^J$ to $\lambda_s^J$.
Our analysis here and in the following sections is left-oriented (i.e., $\lambda$-oriented). A right-sided analogue (i.e., a $\rho$-analogue) is implicitly understood for every result.

\begin{lem}
\label{casicover}
	Let $x,y \in W$, with $x\lhd y$. Then one of the following conditions applies:
	\begin{enumerate}
		\item
        \label{casicover1}
        $\li Jx\lhd \li Jy$ and $\lu Jx=\lu Jy$;
		\item 
        \label{casicover2}
        $\li Jx=\li Jy$ and $\lu Jx\lhd \lu Jy$;
		\item 
        \label{casicover3}
        $\li Jy$ is a proper prefix of $\li Jx$, i.e., there exists $z\in W_J\setminus \{e\}$ such that $\li Jx=\li Jyz$ with $\ell(\li Jx)=\ell(\li Jy)+\ell(z)$, and $\lu Jx<\lu Jy$ but $\lu Jx\not \!\!\lhd \lu Jy$.
	\end{enumerate}
\end{lem}
\begin{proof}
Given a reduced expression of $y$, there exists a letter whose deletion gives a reduced expression of $x$. Choose a reduced expression of $y$ obtained  by concatenating a reduced expression of $\li Jy$ and a reduced expression of $ \lu J y$  (recall  $y=\li Jy \lu Jy$ with $\ell(y) = \ell(\li Jy) + \ell( \lu Jy)$). If the deleted letter belongs to the reduced expression of $\li Jy$, then clearly (\ref{casicover1}) holds.  We may therefore suppose that the deleted letter belongs to the reduced expression of $\lu Jy$, so that $x=\li Jyw$, with $w\lhd \lu Jy$. If $w\in \lu JW$, then $w=\lu Jx$ and (\ref{casicover2}) holds. If $w\notin \lu JW$, then $w=\li Jw \lu Jw$ with $\li Jw\neq e$. In particular,  $\li Jx=\li Jy \li J w$ and $\lu Jx=\lu Jw$, and (\ref{casicover3}) holds.
\end{proof}
For brevity, we say that a pair $(x,y)$ with $x,y\in W$, $x\lhd y$, $M(x)\neq y$, $M(x)\rhd x$ and $M(y)\lhd y$ is a \emph{forbidden edge} for $M$, see Figure \ref{FigForbiddenEdge}.
\begin{figure}[H]
\centering
    \begin{tikzpicture}
    \draw (0,0) node[ thick,draw=black, inner sep=0,outer sep=0.5mm, minimum size=2mm, circle](x) {};
    \draw (-2,2) node[ thick,draw=black, inner sep=0,outer sep=0.5mm, minimum size=2mm, circle](y) {};
    \draw (0,2) node[ thick,draw=black, inner sep=0,outer sep=0.5mm, minimum size=2mm, circle](Mx) {};
    \draw (-2,0) node[ thick,draw=black, inner sep=0,outer sep=0.5mm, minimum size=2mm, circle](My) {};

 \draw[line width=1mm, color=red] (x)--(Mx);
 \draw[line width=1mm, color=red] (y)--(My);
 \draw[thick] (x)--(y);
 
 \node at (0.2,-0.2) {$x$};
 \node at (-2.3,1.8) {$y$};
 \node at (0.6,1.8) {$M(x)$};
 \node at (-2.6,-0.2) {$M(y)$};
 
 \end{tikzpicture}
    \caption{A forbidden edge}
    \label{FigForbiddenEdge}
\end{figure}
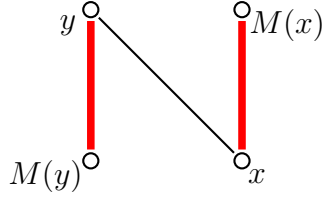 

\begin{prop}\label{forbid} 
Let $M$ be a matching of $W$. Let $u,v \in W$ with $u<v$, $M(v)\lhd v$, and $M(u)\rhd u$. Then $M$ restricts to a special matching of $[u,v]$ if and only if $[u,v]$ contains no forbidden edges.	
\end{prop}
\begin{proof}
It is clear that, if $[u,v]$ contains a forbidden edge, then $M$ does not restrict to a special matching of $[u,v]$. Suppose now that $[u,v]$ contains no forbidden edges and proceed by induction on $\ell(u,v)$ to show that $M$ restricts to a matching of $[u,v]$ and that such restriction is special. 

If $\ell(u,v)\leq 2$, then the result is an easy verification. Suppose $\ell(u,v)>2$. 
Let $w$ be an atom of $[u,v]$ distinct from $M(u)$ (notice that, at this stage, it is not yet guaranteed $M(u)\in [u,v]$). Since $(u,w)$ is not a forbidden edge,  necessarily $M(w)\rhd w$. Moreover, $M(w)\leq v$ by the induction hypothesis applied to the interval $[w,v]$. We need to show $M(u)\lhd M(w)$. Towards a contradiction, suppose otherwise. Since every Bruhat interval of length $2$ in a Coxeter group contains exactly four elements, there exists $y$ in $[u,v]$, with $y\neq w$, such that $u\lhd y\lhd M(w)$. If $M(y)\rhd y$, then $(y,M(w))$ is a forbidden edge. If instead $M(y)\lhd y$, then $(u,y)$ is a forbidden edge. In either case we obtain a contradiction.
\end{proof}
\begin{remark}
According to the proof of Proposition~\ref{forbid}, the non-existence of forbidden edges implies both that $M$ restricts to a matching of $[u,v]$ (i.e., $M$ does not match an element of $[u,v]$ with an element outside $[u,v]$) and that $M$ is special.    
\end{remark}

For all $w\in W$, let $w^s$ denote the projection $w^{\{s\}}$ of $w$ onto $W^{\{s\}}$. To avoid cumbersome notation, we also let $\li Jx^s=(\li Jx)^s$.

\begin{thm}
\label{suff}
Let $s\in J$ and $M=\lambda_s^J$.
Let $u,v \in W $ with $u<v$, $M(u)\rhd u$ and $M(v)\lhd v$. If  $\li Jx^s\leq \li J y ^s$ for all $u\leq x\leq y \leq v$,  then $M$ restricts to a special matching of $[u,v]$. 
\end{thm}
\begin{proof}
	By Proposition \ref{forbid}, it is enough to show that $[u,v]$ contains no forbidden edges. Let  $u\leq x\lhd y \lhd v$ with $x\lhd M(x)\neq y$. In particular, we have $\li J x\lhd \li J x s$ since $x\lhd M(x)$. We also observe  $\li J xs \neq \li J y$. Indeed, if $\li J xs = \li J y$, we would have 
	\[\ell(\li J x)+\ell(\lu J x) +1=\ell(x)+1=\ell(y)=\ell(\li J y)+\ell(\lu J y)=\ell(\li J x) +1 + \ell(\lu J y)\]
	and therefore $\ell(\lu J x)=\ell (\lu J y)$, which implies $\lu J x=\lu J y$  and hence $y=M(x)$, a contradiction. We have to show that $M(y)\rhd y$, i.e. $\li J y s\rhd \li J y$. We consider the three cases of Lemma~\ref{casicover}. In the first case, since $\li J x\lhd \li J y$ and $\li J x \lhd \li J x s \neq \li J y$, the Lifting Property for Coxeter groups yields $\li J y\lhd \li J y s$. In the second case, the assertion is immediate since $\li J y=\li J x$.
The third case cannot occur. Indeed, we would have $\li J y w=\li Jx$ for some $w\in W_J\setminus\{e\}$ with $\ell (\li J y)+ \ell (w)=\ell(\li J x)$.  Since $\li Jx\, s\rhd \li Jx$, we have $\li Jx=\li Jx^s$. Therefore
\[
\li J y^s \leq \li J y < \li J y w =\li J x=\li J x^s,
\]
contradicting the hypothesis.
\end{proof}
\begin{remark}
Note that the condition $\li J x^s\leq \li J y^s$ implies that the third case of Lemma~\ref{casicover} can occur only if $\li J x=\li J y s$, in which case $M(x)\lhd x$ and $M(y)\rhd y$. The converse is not true: if $M(x)\lhd x$ and $M(y)\rhd y$, then we could also be in the first  case of Lemma~\ref{casicover}.
\end{remark}
\begin{remark}
We observe that the condition given in Theorem \ref{suff} is not necessary. Indeed, let $W$ be the Coxeter group of type $A_{n-1}$. Let $J=\{s_2,s_3,s_4\}$ and $s=s_2$ and consider the Bruhat interval $[u,v]$, where $u=[\color{red}2\color{black},\color{red}3\color{black}, 1,\color{red}5\color{black},\color{red}4\color{black}]$ and $v=[\color{red}3\color{black},\color{red}2\color{black},\color{red}4\color{black},\color{red}5\color{black},1]$. Then $\ell(u,v)=2$ and $\lambda_s^J$ is a special matching of $[u,v]$. Nevertheless, letting $y=\lambda_s^J(v)=[\color{red}2\color{black},\color{red}3\color{black},\color{red}4\color{black},\color{red}5\color{black},1]$, we have $u\lhd y$ but $\li J u^s=[1,\color{red}2\color{black},\color{red}3\color{black},\color{red}5\color{black},\color{red}4\color{black}]$ and $\li J y^s=[1,\color{red}2\color{black},\color{red}3\color{black},\color{red}4\color{black},\color{red}5\color{black}]$ so $\li J u^s\nleq \li J y^s$.
(The reason why some entries are black and some other are red will be clarified later and it is not needed at this stage).
\end{remark}
\begin{prop} \label{freccia}
Let $u,v\in W$, with $u\leq v$. Let $M$ be a special matching of $[u,v]$. Then, for all $x,y$ with $u\leq x\lhd y\leq v$, there exists a reflection $t$ such that $M(x)=tM(y)$. 
\end{prop}
\begin{proof}
If $M(x)=y$, then  the result is obvious. If $y\neq M(x)\rhd x$ or $x\neq M(y)\lhd y$, then $M(x)\lhd M(y)$ and the result follows. If $M(x)\lhd x$ and $M(y)\rhd y$, then $M$ restricts to a special matching of $[M(x),M(y)]$, by Lemma~\ref{si_restringe}. By Lemma~\ref{k_corone}, the interval $[M(x),M(y)]$ is a  $k$-crown and, since it admits a special matching, it is actually a $2$-crown or a $3$-crown. Moreover, if $N$ is a special matching of a $3$-crown with top element $\hat 1$ and bottom element $\hat 0$, then $N(\hat 0)\nleq N(\hat 1)$. Therefore, the interval $[M(x),M(y)]$ must necessarily be a $2$-crown and the assertion follows.  
\end{proof}
We will show in Proposition \ref{MdaTaT} the following more general result for a type $A$ Coxeter group $W$: if $[u,v]$ is an interval in $W$  and $x,y\in [u,v]$ satisfy $xy^{-1}\in T$, then also $M(x)M(y)^{-1}\in T$.

\section{First results on special matchings of type $A$ Coxeter groups}
\label{cinque}

Throughout this section, we let $W$ be the Coxeter group of type $A_{n-1}$, i.e.,  the symmetric group on $[n]$  with its standard realization as a Coxeter group. Recall that we use both cycle notation and one-line notation interchangeably. 

     We first note the action of the matchings of the form $\lambda_s^J$ for $W$, and set some notations. Let $J=\{s_i,\ldots,s_{j-1}\}$,  with $i<j\leq n$. We refer to the entries in $\{i,\ldots,j\}$ as the \emph{J-red entries} and the other entries as the \emph{J-black entries}. We write simply red entries and black entries when $J$ is clear from the context. Given $w \in W_J$, we let $w=[w_i,\ldots, w_j]$ be the \emph{$J$-one-line notation} of $w$, i.e., the string obtained from the one-line notation of $w$ by deleting the black entries (which are fixed by $w$). 
     We observe that the $J$-one-line notation of $\li J z$ is obtained from the one-line notation of $z$ by deleting the black entries. Moreover, given $s_k \in J$, the one-line notation of  $\lambda_{s_k}^J(z)$ is obtained from the one-line notation of $z$ by swapping the two red entries $(\li J z)_ k$ and $(\li J z)_{k+1}$, i.e., the $(k-i+1)$-th and the $(k-i+2)$-th red entries from left to right.

     For example, if $n=10$, $i=3$, $j=7$ and $k=4$, then $\lambda_{s_k}^J(z)$ is obtained from $z$ by swapping the second and the third red entries; if
     $z=[{\color{red}4},2,{\color{red}6},
     1,{\color{red}3},{\color{red}7},9,8,{\color{red}5},10]$, then the $J$-one-line notation of $\li J z$ is  $[{\color{red}4},{\color{red}6},
     {\color{red}3},{\color{red}7},{\color{red}5}]$ and $\lambda_{s_k}^J(z)=[{\color{red}4},2,{\color{red}3},
     1,{\color{red}6},{\color{red}7},9,8,{\color{red}5},10]$.

     \begin{remark}
     The map $x\mapsto w_0 x w_0$, where $w_0=[n,\ldots,2,1]$,  is an automorphism of $W$ as a group and as a poset which sends $s_i$ to $s_{n-i}$ for all $i\in[n-1]$.  If $M$ is a matching of an interval $[u,v]$, then we denote by $\tilde M$ the matching of $[w_0 uw_0,w_0 v w_0]$ given by $\tilde M(w_0 x w_0)=w_0 M(x)w_0$. It is clear that $M$ is special if and only if $\tilde M$ is special. We call $\tilde M$ the \emph{companion} matching of $M$.
     Note that if $J=\{s_i,\ldots,s_{j-1}\}$ and we let $\tilde J =w_0Jw_0$ (i.e., $\tilde J=\{s_{n-j+1},\ldots,s_{n-i}\}$),  then for all $x= \li J x \lu J x$ we have $\li {\tilde J}{(w_0 xw_0)}=w_0{\li J x}w_0$ and $\lu {\tilde J}(w_0 xw_0)=w_0 {\lu J x}w_0$. In particular, if $M=\lambda_{s_k}^J$, then
     \[
     \tilde M(w_0 x w_0)=w_0 {M(x)}w_0=w_0{ \li J x s_k \lu J x}w_0=\li {\tilde J}(w_0 xw_0) s_{n-k} \lu {\tilde J} (w_0 xw_0).
     \]
     Hence, the companion matching  $\tilde M$ is $\lambda_{s_{n-k}} ^{\tilde J}$. 
     \end{remark}

Note that, in general, $\lambda_{s_k}^J$ has forbidden edges in $W$ so it is not a special matching of the whole group $W$. For example, let $W$ be of type $A_7$, $J=\{s_3,\ldots,s_6\}$, $k=5$, $x=[\color{red}4\color{black},2,\color{red}7\color{black},\color{red}5\color{black},1,\color{red}6\color{black},8,\color{red}3\color{black}]$, and $y=[\color{red}4\color{black},\color{red}3\color{black},\color{red}7\color{black},\color{red}5\color{black},1,\color{red}6\color{black},8,2]$.  Then $x\lhd y$, $\lambda_{s_k}^J(x)=[\color{red}4\color{black},2,\color{red}7\color{black},\color{red}6\color{black},1,\color{red}5\color{black},8,\color{red}3\color{black}]\rhd x$, and $\lambda_{s_k}^J(y)=[\color{red}4\color{black},\color{red}3\color{black},\color{red}5\color{black},\color{red}7\color{black},1,\color{red}6\color{black},8,2]\lhd y$. Hence, $(x,y)$ is a forbidden edge for $\lambda_{s_k}^J$.
\begin{remark}
\label{J_bassoK^alto}
    Let $J=\{s_i,\ldots,s_{j-1}\}$, let $s_k\in J$, and set $K=\{t\in J:\, ts_k=s_kt\}=J \setminus \{s_{k-1}, s_{k+1}\}$. Then the $J$-one-line notation of $\li J z^K$ is obtained from the $J$-one-line notation $[\li J z_i,\ldots,\li J z_j]$ of $\li J z$  by reordering $\li J z_i,\ldots,\li J z_{k-1}$ in increasing order, reordering $\li J z_k$ and $\li J z_{k+1}$ in increasing order, and reordering $\li J z_{k+2},\ldots,\li J z_j$ in increasing order. In the example above, we have $\li J z=[{\color{red}4}|{\color{red}6},
     {\color{red}3}|{\color{red}7},{\color{red}5}]$ and $\li J z^K=[{\color{red}4}|{\color{red}3},
     {\color{red}6}|{\color{red}5},{\color{red}7}]$. 
The bars indicate the blocks that are reordered.
\end{remark}

The next theorem provides a first characterization of the matchings $\lambda_s^J$ that restrict to special matchings of a given Bruhat interval.

\begin{thm}\label{cha}
Let  $i<j\leq n$ and $J=\{s_i,\ldots,s_{j-1}\}$. Suppose $s\in J$ and $K=\{t\in J:\, ts=st\}$. Let $u<v$ with $\lambda_s^J(u)\rhd u$ and $\lambda_s^J(v)\lhd v$.  Then the following are equivalent:
\begin{enumerate}
    \item $\lambda_s^J$   restricts to a special matching of $[u,v]$;
    \item $\li Jx^K\leq \li J y ^K$, for all $x,y\in [u,v]$ with $x\leq y$.
\end{enumerate} 
\end{thm}
\begin{proof} 
Clearly, the second statement is equivalent to requiring $\li Jx^K\leq \li J y ^K$ for all $x,y\in [u,v]$ with $x\lhd y$. For short, let $M=\lambda_s^J$.

Let $x,y\in [u,v]$, with $x\lhd y$. If $M(x)=y$, then $\li J x =\li J y s$, hence $\li J x^K =\li J y ^K$. Therefore, it suffices to prove that, whenever $x\lhd y$ and $M(x)\neq y$, we have:
\begin{itemize}
    \item if $\li Jx^K\nleq \li J y ^K$ then $M$ does not restrict to a special matching of $[u,v]$;
    \item if $\li Jx^K\leq \li J y ^K$ then $M(x)\leq M(y)$.
\end{itemize}

Consider the three cases of Lemma~\ref{casicover}.
If $\li J x \lhd \li J y$ (and $\lu Jx=\lu Jy$),  then Lemma~\ref{projection} gives $\li J x ^K\leq \li J y^K$. Since $M(x)\neq y$, we have $\li J x \neq \li J ys$. By the Lifting Property,  $\li J x s< \li J y s$ and so $M(x)< M(y)$.

	If $\li J x= \li J y$ and $\lu J x \lhd \lu J y$, then $\li J x^K= \li J y^K$ and $M(x)\lhd M(y)$.

We may, therefore, suppose that we are in the third case of Lemma~\ref{casicover}, when 
there exists $z\in W_J\setminus \{e\}$ such that $\li Jx=\li Jyz$ with $\ell(\li Jx)=\ell(\li Jy)+\ell(z)$, and $\lu Jx<\lu Jy$ but $\lu Jx\not \!\!\lhd \lu Jy$. If $z\in W_K$, then $\li Jx^K =  \li Jy^K$ and $$M(x)= \li Jx s \lu Jx = \li Jy z s \lu Jx = \li Jy s z  \lu Jx < \li Jy s   \lu Jy =M(y)$$ 
where the third equality follows by $zs=sz$ and the inequality follows by  $z \, \lu Jx \lhd    \lu Jy $ (notice that  $z \, \lu Jx \lhd    \lu Jy $ implies  $w z \, \lu Jx<     w \lu Jy $ for all $w\in W_J$ by a repeated use of the Lifting Property). 

We may suppose $z\notin W_K$. In this case, $\li Jy^K <  \li Jx^K$, which means that the triple of sets of red digits in the three blocks considered in Remark~\ref{J_bassoK^alto} differ for $\li J x$ and $\li J y$. 

Let $h$ be such that $s=s_{i+h-1}$ and $r_1,\ldots, r_{h}, r_{h+1}, \ldots, r_{j-i+1}$ be the red digits in the order they appear in $x$. Note $M(x)=(r_{h}, r_{h+1})x$. 
There exists $l\in [j-i+1]$ and a black digit $d$ such that $y=(r_l,d)x$.

Up to passing to the companion matching, one of the following conditions applies:
\begin{enumerate}
    \item $l<h$ and $d$ sits to the right of $r_h$ and (if $h<j-i$) to the left of $r_{h+2}$;
    \item $l=h$ and $d$ sits to the left of $r_{h-1}$; 
    
    \item  $l\leq h$ and $d$ sits to the right of $r_{h+2}$; 
\end{enumerate}

In each of these cases, we apply Proposition~\ref{freccia} and deduce that $M$ is not special by showing  $M(x) M(y)^{-1}\notin T$.
In the first case, $M(y)=(r_l,r_{h+1})y$ and so $M(x) M(y)^{-1}= (r_l,r_{h+1}) (r_l,d)(r_h,r_{h+1})$, which is not a transposition.
Similarly, in the second and third case, when, respectively, $M(y)=(r_{h-1},r_{h+1})y$ and $M(y)=(r_{h+1},r_{h+2})y$.
    \end{proof}
    \begin{remark}
    \label{4.4rmk}
       The proof of the last result shows that $\lambda_{s_k}^J$ is a special matching of the interval $[u,v]$ if and only if, for every covering relation $x\lhd y$ in the interval, at least one of the following is satisfied:
    \begin{enumerate}
        \item $\li J x\lhd \li J y$;
        \item $\{(\li J x)_k,(\li Jx)_{k+1}\} =\{(\li J y)_k, (\li J y)_{k+1}\}$.
        \end{enumerate} 
    \end{remark}

We end this section by showing that any special matching of an interval of length $1$ or $2$ is of the form $\lambda_s^J$ for suitable $s$ and $J$.
From now  on, we identify a subset $J=\{s_i, \ldots, s_{j-1}\}$ of $S$
 with the subset $\{i,\ldots, j\}$ of $[n]$.

\begin{prop}
\label{lunghi1}
	Let $u, v \in W$ with $u \lhd v$. Let $a,b\in [n]$ be such that $v= (a,b)u$. Let $J$ be a subinterval of $[n]$.
    Then $v = \lambda_{s_k}^{J}(u)$ if and only if
    \begin{enumerate}
    \item $J$  contains $a$ and $b$ and does not contain any digit that appears between $a$ and $b$ in  $u$;
 \item $k=\li J u^{-1}(a)$.
 \end{enumerate}
\end{prop}
\begin{proof}
Follows by the analysis at the very beginning of this section. 
\end{proof} 
We illustrate Proposition~\ref{lunghi1} with an example. Consider the Coxeter group of type $A_7$ and the two elements $u=[7,5,2,3,1,8,6,4]$ and $v=[7,5,2,6,1,8,3,4]$, which satisfy $v= (3,6)u$. We let $J=\{3,4,5,6\}$. Note that 3 and 6 occupy positions 2 and 3 among the entries in $\{3,4,5,6\}$. Note also $\li J u=[1,2,5,3,6,4,7,8]$, $\lu J u=[7,3,2,4,1,8,5,6]$, and $v = \lambda_{s_k}^{J}(u)$ for $k=4$. Another possible choice could have been $J'=\{2,3,4,5,6,7\}$ and $k=3$.

In order to achieve our goal to prove that special matchings of all intervals in Coxeter groups of type $A$ are of the form $\lambda_s^J$, we need to show the existence of a common choice for $J$ and $k$ in Proposition~\ref{lunghi1} throughout the interval. The following proposition settles the case of intervals of length 2.
\begin{prop}
	Let $[u,v]$ be an interval of length 2 and $M$ be a special matching of $[u,v]$. Then there exist $J$ and $k$ such that $M(x)= \lambda_{s_k}^{J}(x)$ for all $x\in [u,v]$.
\end{prop}
\begin{proof}
By Lemma~\ref{quadrati}, the four permutations $u,M(u),M(v),v$ agree everywhere but in either $4$ or $3$ positions. 

In the former case,  there exist distinct $a,b,c,d$, with $a<b$, such that
$M(u)=(a,b) u$, $M(v)=(c,d)u$, $v=(a,b)M(v)$. Setting $J=\{a,\ldots,b\}$ and $k=\li J u^{-1}(a)$, we get the assertion.

In the latter case, we call $a,b,c$ the three elements in the three positions where they may differ. We may suppose $u= [\di a \di b \di c \di]$ and, up to considering the companion matching, that either $M(u)=(a,b)u$ or $M(u)=(a,c)u$. Notice that $a\neq \max\{a,b,c\}$ and $c\neq \min\{a,b,c\}$.

If $a<b<c$, then $M(u)=(a,b) u$ and $M(v)=(b,c)u$ because $(a,c)u\not\rhd u$. We have either $M(v)=(a,c)v$, when we let $J=\{a,\ldots,c\}$,  or $M(v)=(a,b)v$, when we let $J=\{a,\ldots,b\}$.
In both cases, we let $k=\li J u^{-1}(a)$, and we get the assertion.

If $a<c<b$, then either $M(u)=(a,b) u$, $M(v)=(a,c)u$, $M(v)=(b,c)v$, 
or
$M(u)=(a,c) u$, $M(v)=(a,b)u$, $M(v)=(a,c)v$. In the former case, we let $J=\{a,\ldots,b\}$. In the latter case, we let  $J=\{a,\ldots,c\}$. In both cases, we let $k=\li J u^{-1}(a)$, and we get the assertion.

If $b<a<c$, then $M(u)=(a,c) u$, $M(v)=(b,c)u$, $M(v)=(a,c)v$.
We let $J=\{a,\ldots,c\}$ and $k=\li J u^{-1}(a)$, and we get the assertion.
\end{proof}

\section{Structural properties of special matchings}
\label{struttura}
In this section, we fix a Bruhat interval $[u,v]$ in the Coxeter group of type $A_{n-1}$, and a special matching $M$ of $[u,v]$.
\begin{lem}
 \label{abovunque}
 Let $a,b \in [n]$ be such that $M(x)= (a,b)  x$ for each atom $x$ of $[u,v]$. Then $M(x)= (a,b)  x$ for each element $x$ of $[u,v]$. 
\end{lem}
\begin{proof}
Note that the hypotheses imply $M(u)= (a,b)  u$. 
Let $x \in [u,v]$. We proceed by induction on $\ell(u,x)$. If  $\ell(u,x) \leq 1 $, the result is true by the hypotheses. 
Let $\ell(u,x) > 1 $. We may suppose $x\lhd M(x)$. Thus, $M$ restricts to a special matching of $[u,M(x)]$. 

If $[u,M(x)]$ is a dihedral interval, then the assertion can be easily verified.
Suppose $[u,M(x)]$ is not a dihedral interval. Then there exist $y_i$ in $[u,M(x)]$, for $i\in \{1,2\}$,  with $y_i \lhd M(x)$, $y_i\neq x$. For $i\in\{1,2\}$, we have $M(y_i)\lhd y_i$ and $M(y_i)\lhd x$ by the Lifting Property, and $M(y_i)= (a,b) y_i$ by the induction hypothesis
(see Figure~\ref{FigLemmaabovunque}; as in all other figures in this work, edge labels are written without parentheses or commas)

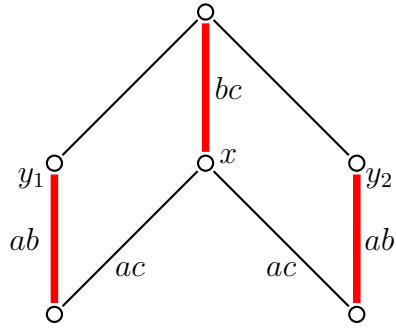
\begin{figure}[H]
\centering
    \begin{tikzpicture}
    \draw (-2,2) node[ thick,draw=black, inner sep=0,outer sep=0.5mm, minimum size=2mm, circle](y) {};
    \draw (2,2) node[ thick,draw=black, inner sep=0,outer sep=0.5mm, minimum size=2mm, circle](x) {};
    \draw (-2,4) node[ thick,draw=black, inner sep=0,outer sep=0.5mm, minimum size=2mm, circle](My) {};
    \draw (0,4) node[ thick,draw=black, inner sep=0,outer sep=0.5mm, minimum size=2mm, circle](x') {};
    \draw (2,4) node[ thick,draw=black, inner sep=0,outer sep=0.5mm, minimum size=2mm, circle](Mx) {};
    \draw (0,6) node[ thick,draw=black, inner sep=0,outer sep=0.5mm, minimum size=2mm, circle](Mx') {};
 \draw[line width=1mm, color=red] (x)--(Mx);
 \draw[line width=1mm, color=red] (y)--(My);
 \draw[line width=1mm, color=red] (x')--(Mx');
 \draw[thick] (x)--(x');
 \draw[thick] (y)--(x');
 \draw[thick] (Mx)--(Mx');
 \draw[thick] (My)--(Mx');
 \node at (2.3,3.8) {$y_2$};
 \node at (-2.3,3.8) {$y_1$};
 \node at (0.3,4.1) {$x$};
 
 \node at (0.3,5) {$bc$};
 
 \node at (1,2.6) {$a c$};
 \node at (-1,2.6) {$a c$};
 \node at (2.3,3) {$ab$};
 \node at (-2.4,3) {$a b$};
 \end{tikzpicture}
    \caption{Proof of Lemma \ref{abovunque}}
    \label{FigLemmaabovunque}
\end{figure} 

If $M(x)\neq (a,b)x$ then, without loss of generality, we can assume $M(x)=(b,c)x$, for some $c\neq a$. By Lemma~\ref{quadrati}, it follows  $M(y_i)=(a,c)x$ for $i\in \{1,2\}$, which is clearly a contradiction. 
\end{proof}
 
 \begin{defn}
We let $i(M)$ be the minimum index $i$ such that there exist $z\in [u,v]$ and $l\in [n]$ such that $M(z)=(i,l)z$. We say that $i(M)$ is the smallest index \emph{moved} by $M$. Similarly, we define $j(M)$ to be the greatest index moved by $M$. We also let 
 \begin{itemize}
     \item $J(M)=\{i(M),i(M)+1,\ldots,j(M)-1,j(M)\}$,
     \item $s(M)=s_{i(M)+h-1}$,
 \end{itemize}
where $h$ is the number of digits in $J(M)$ to the left of $b$ in $u$, and $M(u)=(a,b)u$ with $a<b$.
 
 \end{defn}

 Recall that, given $J\subseteq [n]$, we refer to the numbers in $J$ and to the numbers in $[n]\setminus J$ as, respectively, red and black digits. In this section, having fixed a special matching $M$, the concept of red and black digits refers to the set $J(M)$.
 \begin{prop}\label{ubase}
    Let $a,b \in [n]$ be such that $M(u) = (a,b)u$ and $a<b$.
    Then there are no red digits between $a$ and $b$ in $u$.
    In particular, $$M(u)=\lambda_{s}^{J}(u),$$
    where $J=J(M)$ and $s=s(M)$.
\end{prop}
\begin{proof}
Note that, since $u\lhd M(u)$, each entry that in the one-line notation of $u$ sits between $a$ and $b$ (i.e., each entry in the set $u([u^{-1}(a)+1,u^{-1}(b)-1])$) is smaller than $a$ or greater than $b$.  The assertion is equivalent to the stronger fact that each such entry is smaller than $i$ or greater than $j$.

We use induction on $\ell(u,v)$ to show  

\begin{equation*}
\label{***}
    u([u^{-1}(a),u^{-1}(b)]) \cap [i,a-1] = \emptyset.
\end{equation*} 
This is enough to prove the proposition since it implies  $u([u^{-1}(a),u^{-1}(b)]) \cap [b+1,j] = \emptyset$ by considering the companion matching and we already know  $u([u^{-1}(a),u^{-1}(b)]) \cap [a,b] = \emptyset$.

If $\ell(u,v) \leq 1$, then the assertion is trivial.

Suppose $\ell(u,v) > 1$. If $M(x)= (a,b)  x$ for each atom $x$ of $[u,v]$, then $M(x)= (a,b)  x$ for each element $x$ of $[u,v]$ by Lemma~\ref{abovunque}: hence $i=a$, $j=b$ and the assertion holds not only for $u$ but for each $x$ in $[u,v]$. 

We may suppose that there exists an atom $z$ with $M(z)\neq (a,b)z$. By Lemma \ref{quadrati}, we have 
\begin{itemize}
    \item \label{ao}
$z = (c,d)u$, for some $c\notin\{a,b\}$ and $d\in\{a,b\}$,

   \item \label{ei}
$M(z) = (c,\bar{d})z$, where $\bar{d}\in\{a,b\}\setminus\{d\}$. 

\end{itemize}

If there exists $y$ in $[u,v]\setminus\{v, M(v)\}$ such that $M(y)y^{-1}(i) \neq i$, then we conclude by the induction hypothesis by considering the restriction of $M$ on $[u,\max\{y,M(y)\}]$.  We may suppose that such $y$ does not exist and, hence, $M(v)v^{-1}(i)\neq i$.

We use the induction hypothesis on the restriction of $M$ on $[z,v]$. 
Since $a <b$, either $a<b<c$, or $c<a<b$, or $a<c<b$.

\emph{Case $a<b<c$.} 
In this case, $u^{-1}(c)> u^{-1}(b)$ and $(d,\bar{d})=(b, a)$ since $u\lhd z\lhd M(z)$. By the induction hypothesis, $z([z^{-1}(a),z^{-1}(c)]) \cap [i,a-1] = \emptyset$. Since $z^{-1}(a)= u^{-1}(a)$, $z^{-1}(c)=u^{-1}(b)$, and $z(m)=u(m)$ for all $m\in (z^{-1}(a),z^{-1}(c))$, the assertion is proved.

\emph{Case $c<a<b$.} 
In this case, $u^{-1}(c)< u^{-1}(a)$ and $(d,\bar{d})=(a, b)$ since $u\lhd z\lhd M(z)$.
By the induction hypothesis, $z([z^{-1}(c),z^{-1}(b)]) \cap [i,c-1] = \emptyset$. Furthermore, $z([z^{-1}(c),z^{-1}(b)]) \cap [c,b] = \emptyset$ since $z\lhd M(z)$.
Since $z^{-1}(c)= u^{-1}(a)$, $z^{-1}(b)=u^{-1}(b)$, and $z(m)=u(m)$ for all $m\in (z^{-1}(c),z^{-1}(b))$, the assertion is proved.

\emph{Case $a<c<b$.} 
In this case, $u^{-1}(c) \notin [u^{-1}(a), u^{-1}(b)]$ since $u\lhd M(u)$. If $u^{-1}(c) < u^{-1}(a)$, then  $(d,\bar{d})=(b, a)$ since $u\lhd z\lhd M(z)$ and we may conclude as in the case $a<b<c$.
If $u^{-1}(c) > u^{-1}(b)$,  then  $(d,\bar{d})=(a, b)$ since $u\lhd z\lhd M(z)$ and we may conclude as in the case $c<a<b$.
\end{proof}

Let $x,y\in [u,v]$, with $x\lhd y$ and $y\neq M(x)\rhd x$. Let $M(x)=t_1x$ and $y=t_2x$, with $t_1t_2\neq t_2 t_1$. It is useful to assign a type to such pair $(x,y)$  depending on the local configuration. As established by the next result, Figure~\ref{FigAtoms} (whose notation is motivated by the following sections) shows all possible types, that we call \textnormal{(t1)--(t4)},\textnormal{(t1$^*$)--(t4$^*$)}  (where \textnormal{(tr$^*$)} denotes the companion of \textnormal{(tr)}).
\begin{lem}\label{atoms}
    Let $x,y \in [u,v]$ be such that $x\lhd y=(f,g)x$ and $x\lhd M(x)=(a,b)x$, with $a<b$ . Assume $|\{a,b\}\cap \{f,g\}|=1$. Then the pair $(x,y)$ falls in one of the types \textnormal{(t1)--(t4)} or \textnormal{(t1$^*$)--(t4$^*$)}.
\end{lem}
\begin{proof}
In this proof, we use several times Lemma~\ref{quadrati} without explicit mention.

Suppose $y=(a,f)x$ for some $f\neq \{a,b\}$. If $x=[\di f \di a \di b \di]$, then we are in case (t1) or case (t3). If $x=[\di a \di f \di b\di ]$, then $f>a$ since $y\rhd x$ and $f>b$ since $M(x)\rhd x$. These conditions force $M(y)=(a,b)y$, and we are in case (t4). If $x=[\di a \di b \di f \di ]$, then $a<f<b$ since $y\rhd x$, and $M(y)=(f,b)y$ since $(a,b)y\lhd y$; hence we are in case (t2).

The case $y=(b,f)x$ is the companion of the previous case, and we obtain types (t1$^*$)--(t4$^*$). 

\end{proof}

   \begin{figure}[h] 
\centering
    \begin{tikzpicture}
    \draw[thick](0,9)--(15,9);
    \draw[thick](0,12)--(15,12);
    \draw[thick](0,15)--(15,15);
    \draw[thick](0,18)--(15,18);
    \draw[thick](0,21)--(15,21);
    \draw[thick](0,9)--(0,21);
    \draw[thick](15,9)--(15,21);
    \draw[thick](7.5,9)--(7.5,21);
      \draw (1.8,18.3) node[ thick,draw=black, inner sep=0,outer sep=0.5mm, minimum size=2mm, circle](u) {};
      \draw (3,19.5) node[thick, draw=black, inner sep=0,outer sep=0.5mm, minimum size=2mm, circle](Mu) {};
      \draw (0.5,19.5) node[thick, draw=black, inner sep=0,outer sep=0.5mm, minimum size=2mm, circle](x) {};
      \draw (1.8,20.7) node[thick, draw=black, inner sep=0,outer sep=0.5mm, minimum size=2mm, circle](Mx) {};

\draw[line width=1mm, color=red] (u)--(Mu);
\draw[thick] (u)--(x);
\draw[line width=1mm, color=red] (x)--(Mx);
\draw[thick] (Mx)--(Mu);

\node at (4,20.65) {(t1)};

\node at (1.5,18.2) {$x$};
\node at (0.2,19.6) {$y$};
\node at (2.7,18.8) {$ab$};
\node at (0.8,18.8) {$a_1a$};
\node at (0.9,20.3) {$a_1b$};
\node at (5.5,19.6) {$x=[\di a_1\di  a \di b\di] $,};
\node at (5.5,19) {with $a_1<a<b$.};

\node at (4,17.65) {(t2)};

      \draw (1.8,15.3) node[ thick,draw=black, inner sep=0,outer sep=0.5mm, minimum size=2mm, circle](u2) {};
      \draw (3,16.5) node[thick, draw=black, inner sep=0,outer sep=0.5mm, minimum size=2mm, circle](Mu2) {};
      \draw (0.5,16.5) node[thick, draw=black, inner sep=0,outer sep=0.5mm, minimum size=2mm, circle](x2) {};
      \draw (1.8,17.7) node[thick, draw=black, inner sep=0,outer sep=0.5mm, minimum size=2mm, circle](Mx2) {};

\draw[line width=1mm, color=red] (u2)--(Mu2);
\draw[thick] (u2)--(x2);
\draw[line width=1mm, color=red] (x2)--(Mx2);
\draw[thick] (Mx2)--(Mu2);

\node at (1.5,15.2) {$x$};
\node at (0.2,16.6) {$y$};
\node at (2.7,15.8) {$ab$};
\node at (0.8,15.8) {$a \alpha$};
\node at (0.9,17.3) {$\alpha b$};
\node at (5.5,16.6) {$x=[\di a\di  b \di \alpha \di ]$,};
\node at (5.5,16) {with $a<\alpha<b$.};
\node at (4,14.65) {(t3)};

      \draw (1.8,12.3) node[ thick,draw=black, inner sep=0,outer sep=0.5mm, minimum size=2mm, circle](u3) {};
      \draw (3,13.5) node[thick, draw=black, inner sep=0,outer sep=0.5mm, minimum size=2mm, circle](Mu3) {};
      \draw (0.5,13.5) node[thick, draw=black, inner sep=0,outer sep=0.5mm, minimum size=2mm, circle](x3) {};
      \draw (1.8,14.7) node[thick, draw=black, inner sep=0,outer sep=0.5mm, minimum size=2mm, circle](Mx3) {};

\draw[line width=1mm, color=red] (u3)--(Mu3);
\draw[thick] (u3)--(x3);
\draw[line width=1mm, color=red] (x3)--(Mx3);
\draw[thick] (Mx3)--(Mu3);

\node at (1.5,12.2) {$x$};
\node at (0.2,13.6) {$y$};
\node at (2.7,12.8) {$ab$};
\node at (0.8,12.8) {$pa$};
\node at (0.9,14.3) {$ab$};
\node at (5.5,13.6) {$x=[\di p\di  a \di b\di ]$,};
\node at (5.5,13) {with $p<a<b$.};

\node at (4,11.65) {(t4)};
      \draw (1.8,9.3) node[ thick,draw=black, inner sep=0,outer sep=0.5mm, minimum size=2mm, circle](u4) {};
      \draw (3,10.5) node[thick, draw=black, inner sep=0,outer sep=0.5mm, minimum size=2mm, circle](Mu4) {};
      \draw (0.5,10.5) node[thick, draw=black, inner sep=0,outer sep=0.5mm, minimum size=2mm, circle](x4) {};
      \draw (1.8,11.7) node[thick, draw=black, inner sep=0,outer sep=0.5mm, minimum size=2mm, circle](Mx4) {};

\draw[line width=1mm, color=red] (u4)--(Mu4);
\draw[thick] (u4)--(x4);
\draw[line width=1mm, color=red] (x4)--(Mx4);
\draw[thick] (Mx4)--(Mu4);

\node at (1.5,9.2) {$x$};
\node at (0.2,10.6) {$y$};
\node at (2.7,9.8) {$ab$};
\node at (0.8,9.8) {$ac$};
\node at (0.9,11.3) {$ab$};
\node at (5.5,10.6) {$x=[\di a\di  c \di b\di] $,};
\node at (5.5,10) {with $a<b<c$.};

\node at (11.5,20.65) {(t1$^*$)};
      \draw (9.3,18.3) node[ thick,draw=black, inner sep=0,outer sep=0.5mm, minimum size=2mm, circle](u7) {};
      \draw (10.5,19.5) node[thick, draw=black, inner sep=0,outer sep=0.5mm, minimum size=2mm, circle](Mu7) {};
      \draw (8,19.5) node[thick, draw=black, inner sep=0,outer sep=0.5mm, minimum size=2mm, circle](x7) {};
      \draw (9.3,20.7) node[thick, draw=black, inner sep=0,outer sep=0.5mm, minimum size=2mm, circle](Mx7) {};

\draw[line width=1mm, color=red] (u7)--(Mu7);
\draw[thick] (u7)--(x7);
\draw[line width=1mm, color=red] (x7)--(Mx7);
\draw[thick] (Mx7)--(Mu7);

\node at (9,18.2) {$x$};
\node at (7.7,19.6) {$y$};
\node at (10.2,18.8) {$ab$};
\node at (8.3,18.8) {$bb_1$};
\node at (8.4,20.3) {$ab_1$};
\node at (13,19.6) { $x=[\di a\di  b \di b_1\di ]$,};
\node at (13,19) {with $a<b<b_1$.};

\node at (11.5,17.65) {(t2$^*$)};
      \draw (9.3,15.3) node[ thick,draw=black, inner sep=0,outer sep=0.5mm, minimum size=2mm, circle](u8) {};
      \draw (10.5,16.5) node[thick, draw=black, inner sep=0,outer sep=0.5mm, minimum size=2mm, circle](Mu8) {};
      \draw (8,16.5) node[thick, draw=black, inner sep=0,outer sep=0.5mm, minimum size=2mm, circle](x8) {};
      \draw (9.3,17.7) node[thick, draw=black, inner sep=0,outer sep=0.5mm, minimum size=2mm, circle](Mx8) {};

\draw[line width=1mm, color=red] (u8)--(Mu8);
\draw[thick] (u8)--(x8);
\draw[line width=1mm, color=red] (x8)--(Mx8);
\draw[thick] (Mx8)--(Mu8);

\node at (9,15.2) {$x$};
\node at (7.7,16.6) {$y$};
\node at (10.2,15.8) {$ab$};
\node at (8.3,15.8) {$\alpha b$};
\node at (8.4,17.3) {$a \alpha$};
\node at (13,16.6) {$x=[\di \alpha\di  a \di b\di ]$,};
\node at (13,16) {with $a<\alpha<b$.};

\node at (11.5,14.65) {(t3$^*$)};
      \draw (9.3,12.3) node[ thick,draw=black, inner sep=0,outer sep=0.5mm, minimum size=2mm, circle](u9) {};
      \draw (10.5,13.5) node[thick, draw=black, inner sep=0,outer sep=0.5mm, minimum size=2mm, circle](Mu9) {};
      \draw (8,13.5) node[thick, draw=black, inner sep=0,outer sep=0.5mm, minimum size=2mm, circle](x9) {};
      \draw (9.3,14.7) node[thick, draw=black, inner sep=0,outer sep=0.5mm, minimum size=2mm, circle](Mx9) {};

\draw[line width=1mm, color=red] (u9)--(Mu9);
\draw[thick] (u9)--(x9);
\draw[line width=1mm, color=red] (x9)--(Mx9);
\draw[thick] (Mx9)--(Mu9);

\node at (9,12.2) {$x$};
\node at (7.7,13.6) {$y$};
\node at (10.2,12.8) {$ab$};
\node at (8.3,12.8) {$bc$};
\node at (8.4,14.3) {$a b$};
\node at (13,13.6) {$x=[\di a\di  b \di c\di ]$,};
\node at (13,13) {with $a<b<c$.};

\node at (11.5,11.65) {(t4$^*$)};
      \draw (9.3,9.3) node[ thick,draw=black, inner sep=0,outer sep=0.5mm, minimum size=2mm, circle](u10) {};
      \draw (10.5,10.5) node[thick, draw=black, inner sep=0,outer sep=0.5mm, minimum size=2mm, circle](Mu10) {};
      \draw (8,10.5) node[thick, draw=black, inner sep=0,outer sep=0.5mm, minimum size=2mm, circle](x10) {};
      \draw (9.3,11.7) node[thick, draw=black, inner sep=0,outer sep=0.5mm, minimum size=2mm, circle](Mx10) {};

\draw[line width=1mm, color=red] (u10)--(Mu10);
\draw[thick] (u10)--(x10);
\draw[line width=1mm, color=red] (x10)--(Mx10);
\draw[thick] (Mx10)--(Mu10);

\node at (9,9.2) {$x$};
\node at (7.7,10.6) {$y$};
\node at (10.2,9.8) {$ab$};
\node at (8.3,9.8) {$pb$};
\node at (8.4,11.3) {$a b$};
\node at (13,10.6) {$x=[\di a\di  p \di b\di ]$,};
\node at (13,10) {with $p<a<b$.};

\end{tikzpicture}
    \caption{Local configurations of a special matching}
    \label{FigAtoms}
\end{figure}
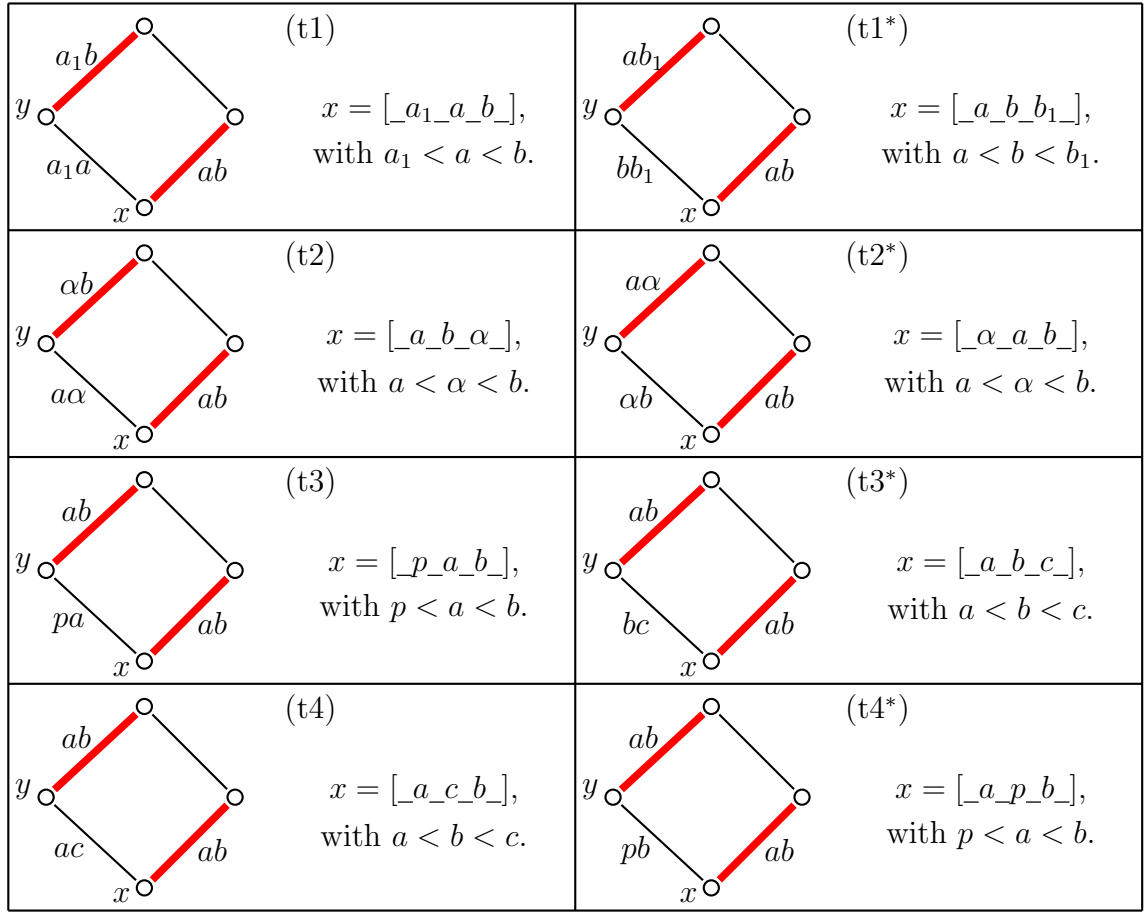  

Next result will be extremely useful.

\begin{lem}\label{absednero} Let $J=J(M)$ and $s=s(M)$. Let $x,y\in [u,v]$ and $a,b,c,d\in [n]$, with $d$ black, be such that $y=(c,d)x$ and $M(x)=(a,b)x$.
Then $M(y)=(a,b)y$.
\end{lem}
\begin{proof}

Up to considering the interval $[vw_0,uw_0]$ and the special matching $M'$ of it given by $M'(xw_0)=M(x)w_0$, we can suppose $x<y$, and up to considering the companion of $M$, we can also suppose $c<d$. 

We proceed by induction on $\ell(x,y)$. Let $\ell(x,y)=1$. If $M(x)\rhd x$ then $M(y)\rhd y$. Since $d$ is black, necessarily $M(y)=(a,b)y$ (note that $d\notin \{a,b\}$ since $d$ is black, while we may have $c\in \{a,b\}$). If $M(x)\lhd x$ and $M(y)\lhd y$, then the argument is similar. If $M(x)\lhd x$ and $M(y)\rhd y$, then $[M(x),M(y)]$ is a 2-crown and, since $d$ is black, we have $M(y)=(a,b)y$.

Now suppose $\ell(x,y)=2s+1>1$, and let $c_1,\ldots,c_s$ be such that $c<c_i<d$ for each $i\in [s]$ and $x=[\di c \di c_1 \ldots c_s\di d \di ]$. Let $x_0=x$ and define recursively $x_{i+1}=(c_{s-i},d)x_i$ (see Figure~\ref{lemabsednero} for an illustration with $s=3$; note that, although in Figure~\ref {lemabsednero} we have depicted $M(z_i)\rhd z_i$, 
this is not necessarily the case) for each $i\in [s]$, where we set $c_0=c$. Note that $x_{i+1}\rhd x_i$ and 
\[
x_{s+1}=[\di d \di c \di c_1 \ldots c_s \di ].
\]
Note also $x_{s+1}\leq y$ by the Tableau Criterion, and so in particular $x_{s+1}\in [u,v]$; by the first part of the proof, $M(x_i)=(a,b)x_i$ for each $i\in [s+1]$.

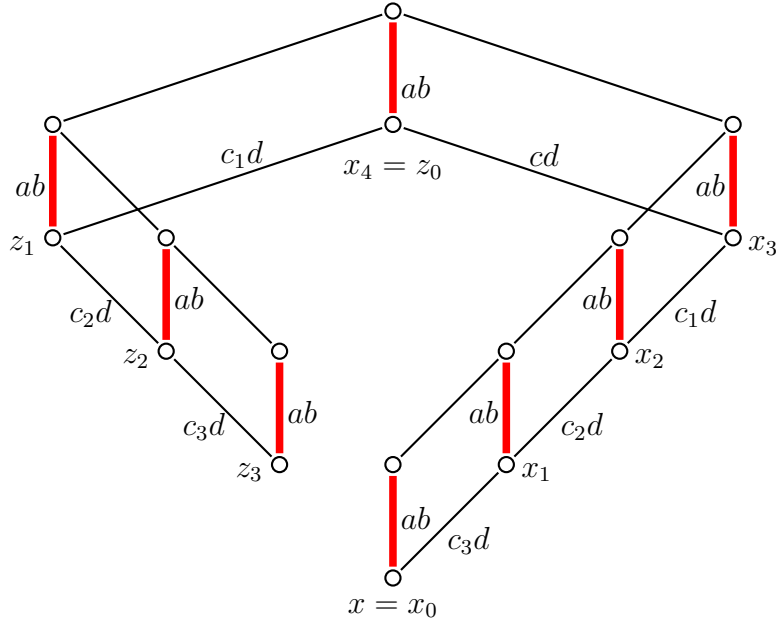
\begin{figure}[h] 
\centering
    \begin{tikzpicture}
 \draw (0,0) node[ thick,draw=black, inner sep=0,outer sep=0.5mm, minimum size=2mm, circle](x) {};
 \draw (0,1.5) node[ thick,draw=black, inner sep=0,outer sep=0.5mm, minimum size=2mm, circle](Mx) {};
 \draw (-1.5,1.5) node[ thick,draw=black, inner sep=0,outer sep=0.5mm, minimum size=2mm, circle](z3) {};
 \draw (-1.5,3) node[ thick,draw=black, inner sep=0,outer sep=0.5mm, minimum size=2mm, circle](Mz3) {};
 \draw (-3,3) node[ thick,draw=black, inner sep=0,outer sep=0.5mm, minimum size=2mm, circle](z2) {};
\draw (-3,4.5) node[ thick,draw=black, inner sep=0,outer sep=0.5mm, minimum size=2mm, circle](Mz2) {};
\draw (-4.5,4.5) node[ thick,draw=black, inner sep=0,outer sep=0.5mm, minimum size=2mm, circle](z1) {};
\draw (-4.5,6) node[ thick,draw=black, inner sep=0,outer sep=0.5mm, minimum size=2mm, circle](Mz1) {};
\draw (0,6) node[ thick,draw=black, inner sep=0,outer sep=0.5mm, minimum size=2mm, circle](x4) {};
\draw (0,7.5) node[ thick,draw=black, inner sep=0,outer sep=0.5mm, minimum size=2mm, circle](Mx4) {};
 \draw[line width=1mm, color=red] (x)--(Mx);
 \draw[line width=1mm, color=red] (z3)--(Mz3);
 \draw[line width=1mm, color=red] (z2)--(Mz2);
 \draw[line width=1mm, color=red] (z1)--(Mz1);
\draw[line width=1mm, color=red] (x4)--(Mx4);
\draw[thick] (z1)--(x4);
\draw[thick] (z2)--(z1);
\draw[thick] (z3)--(z2);
\draw[thick] (Mz3)--(Mz2);
\draw[thick] (Mz1)--(Mz2);

\node at (0,-0.4) {$x=x_0$};
\node at (1.9,1.4) {$x_1$};
\node at (3.4,2.9) {$x_{2}$};
\node at (4.9,4.4) {$x_{3}$};
\node at (-1.9,1.4) {$z_3$};
\node at (-3.4,2.9) {$z_{2}$};
\node at (-4.9,4.4) {$z_{1}$};
\node at (0,5.4) {$x_4=z_0$};

\node at (0.3,0.8) {$ab$};
\node at (-1.2,2.2) {$ab$};
\node at (-2.7,3.7) {$ab$};
\node at (-4.8,5.2) {$ab$};

\node at (1,0.5) {$c_3d$};
\node at (2.5,2) {$c_2d$};
\node at (4,3.5) {$c_1d$};
\node at (2,5.6) {$cd$};

\node at (-2.5,2) {$c_3d$};
\node at (-4,3.5) {$c_2d$};
\node at (-2,5.6) {$c_1d$};
\node at (1.2,2.2) {$ab$};
\node at (2.7,3.7) {$ab$};
\node at (4.2,5.2) {$ab$};
\node at (0.3,6.5) {$ab$};
\draw (1.5,1.5) node[ thick,draw=black, inner sep=0,outer sep=0.5mm, minimum size=2mm, circle](x1) {};
 \draw (1.5,3) node[ thick,draw=black, inner sep=0,outer sep=0.5mm, minimum size=2mm, circle](Mx1) {};
 \draw (3,3) node[ thick,draw=black, inner sep=0,outer sep=0.5mm, minimum size=2mm, circle](x2) {};
\draw (3,4.5) node[ thick,draw=black, inner sep=0,outer sep=0.5mm, minimum size=2mm, circle](Mx2) {};
\draw (4.5,4.5) node[ thick,draw=black, inner sep=0,outer sep=0.5mm, minimum size=2mm, circle](x3) {};
\draw (4.5,6) node[ thick,draw=black, inner sep=0,outer sep=0.5mm, minimum size=2mm, circle](Mx3) {};
 \draw[line width=1mm, color=red] (x1)--(Mx1);
 \draw[line width=1mm, color=red] (x2)--(Mx2);
 \draw[line width=1mm, color=red] (x3)--(Mx3);
\draw[thick] (x)--(x1);
\draw[thick] (x1)--(x2);
\draw[thick] (x2)--(x3);
\draw[thick] (Mx)--(Mx1);
\draw[thick] (Mx1)--(Mx2);
\draw[thick] (Mx2)--(Mx3);
\draw[thick] (x3)--(x4);
\draw[thick] (Mx3)--(Mx4);
\draw[thick] (Mz1)--(Mx4);

\end{tikzpicture}
    \caption{Proof of Lemma~\ref{absednero} for $s=3$}
    \label{lemabsednero}
\end{figure} 

Now let $z_0=x_{s+1}$, and $z_i=(c_i,d)z_{i-1}$ for each $i\in [s]$. Note that $z_{i}\lhd z_{i-1}$, and, by the first part of the proof, $M(z_i)=(a,b)z_i$. Notice 
\[
z_s=[\di c_1 \di c \di c_2 \ldots c_s \di d \di ]
\]
and $z_s=(c_1,c)x\rhd x$. Letting $y'=(c,d) z_s$, we have $y'=[\di c_1 \di d \di c_2 \ldots c_s \di c \di ]=(c_1,d)y\lhd y$. In particular, $\ell(z_s,y')=\ell(x,y)-2$ and so, by induction hypothesis, $M(y')=(a,b)y'$.   Since $y=(c_1,d)y'$,  the first part of the proof implies $M(y)=(a,b)y.$
\end{proof}

The following result is crucial for the classification of special matchings.

\begin{prop}[Lemma of the basic chains]\label{chains_ij}
    Let $a,b \in [n]$ be such that $M(u)=(a,b)u$, with $a<b$. Then there exist $a_1,\ldots,a_r, b_1, \ldots, b_s\in [n]$ with \[i(M)=a_r<\cdots<a_1<a<b<b_1<\cdots <b_s=j(M)\] such that the following conditions are satisfied:
    \begin{enumerate}
\item $u=[\di a_r \ldots a_1 \di a \di b \di b_1 \ldots b_s\di ]$; 
\item letting $a_0=a$, $u_0=u$ and $u_m=(a_{m-1},a_m) u_{m-1}$ for each $m \in [r]$, we have 
\begin{itemize}
    \item $u_m\in [u,v]$, 
    \item $u_m\rhd u_{m-1}$, 
    \item $M(u_m)=(a_m,b)u_m$, 
    \end{itemize}
    for each $m\in [r]$;
\item letting $b_0=b$, $u'_0=u$ and $u_m'=(b_{m-1},b_m)u_{m-1}'$ for each $m\in [s]$, we have 
\begin{itemize}
    \item $u'_m\in [u,v]$, 
    \item $u'_m\rhd u'_{m-1}$,  
    \item $M(u_m)=(a,b_m)u_m$,
     \end{itemize}
      for each $m\in [s]$.
    \end{enumerate}
\end{prop}
The statement of Proposition~\ref{chains_ij} is depicted in Figure~\ref{fig_chains_ij}.

\begin{proof}
  Up to considering the companion matching, it is enough to prove the existence of $a_1,\ldots,a_r$. We proceed by induction on $\ell(u,v)$. Let $z\in [u,v]$ be such that $M(z)=(i(M),k)z$, for some $k$. If $z\notin \{v,M(v)\}$,
    then the result clearly follows by induction. We can suppose $z=v$. 
    
    We first claim that, if $a\neq i(M)$, then $[u,v]$ has an atom $y$ such that the pair $(u,y)$ is of type (t1), see  Figure~\ref{FigAtoms}. 

    Suppose that there is an atom $x$ such that $M(x)=(a,f)x$ with $f\neq b$, i.e., $(u,x)$ is of type (t1$^*$) or (t2$^*$). Then, by induction hypothesis, there exist $a_1 \in [a]$ and $x'\in [u,v]$ such that $x'\rhd x$, and $M(x')=(a_1,f) x'$. The interval $[u,M(x')]$ is a 3-crown. Call $y$ the atom distinct from $x$ and $M(u)$. By applying Lemma~\ref{quadrati} to the interval $[u,x']$, we deduce $y=(a,a_1)u$. We also have $M(y)\neq (a,b)y$, by Lemma~\ref{quadrati} applied to $[y,M(x')]$, since $M(x')=(a_1,f)x'$. We conclude that $M(y)=(a_1,b) y$ by Lemma~\ref{quadrati} applied to $[u,M(y)]$. Hence $(u,y)$ is of type (t1).

    If $(u,x)$ is of type (t2) for every atom $x\neq M(u)$, then all covering relations of $[u,v]$ are given by transpositions involving numbers in $\{a,\ldots,b\}$. Hence $i(M)=a$, and there is nothing to prove. 
    
    By the classification in Figure~\ref{FigAtoms}, it remains to treat the cases when there is an atom $x$ with $M(x)=(a,b)x$. By induction hypothesis, there exist $a_1 \in [a]$ and $x'\in [u,v]$ such that $x'\rhd x$, and $M(x')=(a_1,b) x'$. The interval $[u,M(x')]$ is necessarily a 3-crown. Call $y$ the atom distinct from $x$ and $M(u)$. By the first part of the proof, we may suppose that $(u,y)$ is not of type (t1$^*$) or (t2$^*$). Moreover, $M(y)\neq (a,b) y$ by applying Lemma~\ref{abovunque} to $[u,M(x')]$. So $(u,y)$ is necessarily of type (t1) or (t2), by the classification in Figure~\ref{FigAtoms},
    Actually, since $a_1<a<\alpha$, the pair $(u,y)$  cannot be of type (t2) by   Lemma~\ref{quadrati} applied to $[u,x']$,  see Figure~\ref{ProofLemmaChains}.
We conclude that $y$ is necessarily of type (t1). The claim is proved and we have achieved the first step obtaining the element $u_1$.
    
\begin{figure}[h]
\centering
    \begin{tikzpicture}
    \draw (0,0) node[ thick,draw=black, inner sep=0,outer sep=0.5mm, minimum size=2mm, circle](u) {};
    \draw (-2,2) node[ thick,draw=black, inner sep=0,outer sep=0.5mm, minimum size=2mm, circle](y) {};
    \draw (0,2) node[ thick,draw=black, inner sep=0,outer sep=0.5mm, minimum size=2mm, circle](Mu) {};
    \draw (2,2) node[ thick,draw=black, inner sep=0,outer sep=0.5mm, minimum size=2mm, circle](x) {};
    \draw (-2,4) node[ thick,draw=black, inner sep=0,outer sep=0.5mm, minimum size=2mm, circle](My) {};
    \draw (0,4) node[ thick,draw=black, inner sep=0,outer sep=0.5mm, minimum size=2mm, circle](x') {};
    \draw (2,4) node[ thick,draw=black, inner sep=0,outer sep=0.5mm, minimum size=2mm, circle](Mx) {};
    \draw (0,6) node[ thick,draw=black, inner sep=0,outer sep=0.5mm, minimum size=2mm, circle](Mx') {};
    \draw[line width=1mm, color=red] (u)--(Mu);
 \draw[line width=1mm, color=red] (x)--(Mx);
 \draw[line width=1mm, color=red] (y)--(My);
 \draw[line width=1mm, color=red] (x')--(Mx');
 \draw[thick] (u)--(x);
 \draw[thick] (u)--(y);
 \draw[thick] (x)--(x');
 \draw[thick] (y)--(x');
 \draw[thick] (Mu)--(Mx);
 \draw[thick] (Mu)--(My);
 \draw[thick] (Mx)--(Mx');
 \draw[thick] (My)--(Mx');
 \node at (0.2,-0.2) {$u$};
 \node at (2.2,1.8) {$x$};
 \node at (-2.2,1.8) {$y$};
 \node at (0.3,4.2) {$x'$};
 \node at (0.3,1.2) {$ab$};
 \node at (2.3,3) {$ab$};
 \node at (0.3,5) {$a_1b$};
 \node at (-1.4,1) {$a \alpha$};
 \node at (-2.4,3) {$\alpha b$};
 \node at (0.3,3.3) {$a_1 a$};
 \node at (1.2,5.2) {$a_1 a$};
 \node at (-0.8,3.6) {$a_1 \alpha$};
 \node at (-1.1,5.3) {$a_1 \alpha$};
 \node at (-0.7,2.4) {$a \alpha$};
 
 \end{tikzpicture}
    \caption{Proof of Proposition \ref{chains_ij}}
    \label{ProofLemmaChains}
\end{figure}
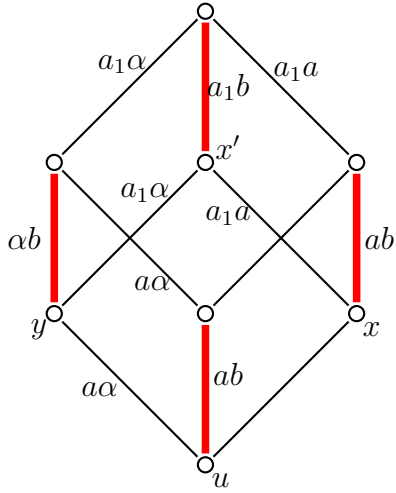

    Another claim that we prove is the following: if there is a sequence $a_m<\cdots<a_1<a_0=a$ for some $m\geq 1$ and $a_m> i(M)$, with $u=[\di a_m \ldots a_1 \di a \di b \di]$ which satisfies condition (2) of the statement, then one of the following two conditions apply:
    \begin{itemize}
      \item[(i)] there exists $a_{m+1}<a_m$ such that the sequence $a_{m+1}<\cdots <a_1<a_0$ is such that $u=  [\di a_{m+1} \ldots a_1 \di a \di b \di]$ and satisfies condition (2) of the statement;
      \item[(ii)] there exists $a_m'<a_m$ such that the sequence $a'_{m}<a_{m-1}<\cdots <a_1<a_0$ is such that $u=  [\di a'_{m} \di a_{m-1} \ldots a_1 \di a \di b \di]$ and satisfies condition (2) of the statement.
    \end{itemize}
    This will conclude the proof.
    
    So we have that $u_{m-1}=[\di a_m \di a_{m-1} \di b \di] $ (with $M(u_{m-1})=(a_{m-1},b)u_{m-1}$) and $u_m=[\di a_{m-1} \di a_m \di b \di ]$ with $M(u_m)=(a_m,b)u_m$. By the first part of the proof, there exists
    $p \in [a_m]$ with either $u_m=[\di a_{m-1} \di  p \di  a_m \di b \di ]$ or $u_m=[\di p \di a_{m-1} \di a_m \di b\di ]$ such that, letting $z=(a_m,p)u_m$, we have $z\rhd u_m$ and $M(z)=(p,b)z$.
    In the latter case, we let $a_{m+1}=p$ and we are in case (i).
    In the former case, we let $a_m'=p$ and we have $u_{m-1}=[\di a_m \di a'_{m} \di  a_{m-1} \di  b \di ]$. Consider the element  $u_m'= [\di a_m \di a_{m-1} \di  a'_{m} \di  b \di ]$. Observe $u_m'\rhd u_{m-1}$ since $u_{m-1}\lhd u_m$ and $z\rhd u_m$.  Since $z=(a_{m-1},a_m)u_m'$, necessarily  $M(u_m')=(a'_{m},b)u_m'$, thus yielding case (ii). 
\end{proof}
\color{black}

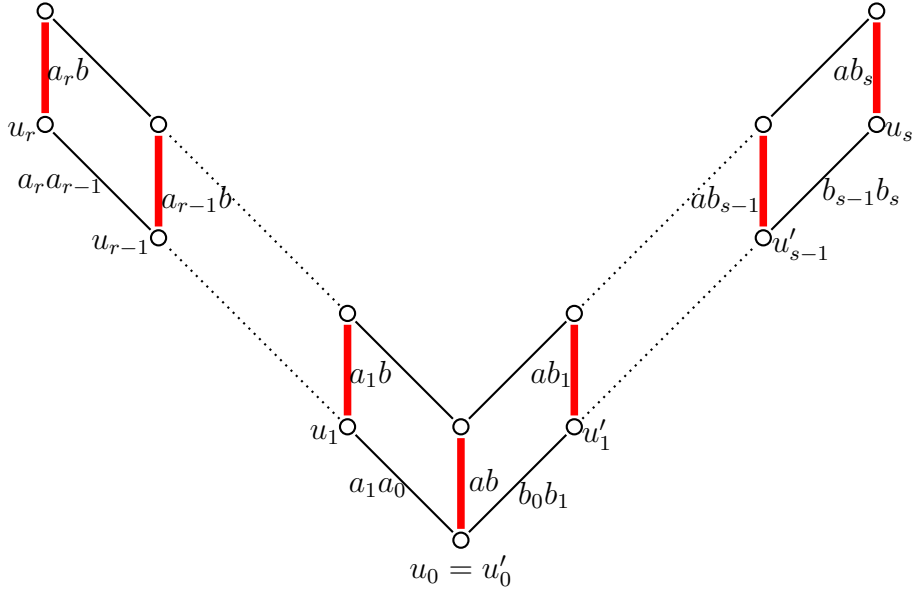
\begin{figure}[h] 
\centering
    \begin{tikzpicture}
 \draw (0,0) node[ thick,draw=black, inner sep=0,outer sep=0.5mm, minimum size=2mm, circle](u) {};
 \draw (0,1.5) node[ thick,draw=black, inner sep=0,outer sep=0.5mm, minimum size=2mm, circle](Mu) {};
 \draw (-1.5,1.5) node[ thick,draw=black, inner sep=0,outer sep=0.5mm, minimum size=2mm, circle](u1) {};
 \draw (-1.5,3) node[ thick,draw=black, inner sep=0,outer sep=0.5mm, minimum size=2mm, circle](Mu1) {};
 \draw (-4,4) node[ thick,draw=black, inner sep=0,outer sep=0.5mm, minimum size=2mm, circle](ur1) {};
\draw (-4,5.5) node[ thick,draw=black, inner sep=0,outer sep=0.5mm, minimum size=2mm, circle](Mur1) {};
\draw (-5.5,5.5) node[ thick,draw=black, inner sep=0,outer sep=0.5mm, minimum size=2mm, circle](ur) {};
\draw (-5.5,7) node[ thick,draw=black, inner sep=0,outer sep=0.5mm, minimum size=2mm, circle](Mur) {};
 \draw[line width=1mm, color=red] (u)--(Mu);
 \draw[line width=1mm, color=red] (u1)--(Mu1);
 \draw[line width=1mm, color=red] (ur1)--(Mur1);
 \draw[line width=1mm, color=red] (ur)--(Mur);
\draw[thick] (u)--(u1);
\draw[thick] (ur1)--(ur);
\draw[thick] (Mu)--(Mu1);
\draw[thick] (Mur1)--(Mur);
\draw[thick, dotted] (u1)--(ur1);
\draw[thick, dotted] (Mu1)--(Mur1);
\node at (0,-0.4) {$u_0=u_0'$};
\node at (-1.8,1.4) {$u_1$};
\node at (-4.5,3.9) {$u_{r-1}$};
\node at (-5.8,5.4) {$u_r$};
\node at (-1.1,0.7) {$a_1a_0$};
\node at (-5.3,4.7) {$a_ra_{r-1}$};
\node at (0.3,0.8) {$ab$};
\node at (-1.2,2.2) {$a_1b$};
\node at (-3.5,4.5) {$a_{r-1}b$};
\node at (-5.2,6.2) {$a_{r}b$};

\draw (1.5,1.5) node[ thick,draw=black, inner sep=0,outer sep=0.5mm, minimum size=2mm, circle](u1') {};
 \draw (1.5,3) node[ thick,draw=black, inner sep=0,outer sep=0.5mm, minimum size=2mm, circle](Mu1') {};
 \draw (4,4) node[ thick,draw=black, inner sep=0,outer sep=0.5mm, minimum size=2mm, circle](ur1') {};
\draw (4,5.5) node[ thick,draw=black, inner sep=0,outer sep=0.5mm, minimum size=2mm, circle](Mur1') {};
\draw (5.5,5.5) node[ thick,draw=black, inner sep=0,outer sep=0.5mm, minimum size=2mm, circle](ur') {};
\draw (5.5,7) node[ thick,draw=black, inner sep=0,outer sep=0.5mm, minimum size=2mm, circle](Mur') {};
 \draw[line width=1mm, color=red] (u1')--(Mu1');
 \draw[line width=1mm, color=red] (ur1')--(Mur1');
 \draw[line width=1mm, color=red] (ur')--(Mur');
\draw[thick] (u)--(u1');
\draw[thick] (ur1')--(ur');
\draw[thick] (Mu)--(Mu1');
\draw[thick] (Mur1')--(Mur');
\draw[thick, dotted] (u1')--(ur1');
\draw[thick, dotted] (Mu1')--(Mur1');
\node at (1.8,1.4) {$u'_1$};
\node at (4.5,3.9) {$u'_{s-1}$};
\node at (5.8,5.4) {$u_s$};
\node at (1.1,0.6) {$b_0b_1$};
\node at (5.3,4.6) {$b_{s-1}b_s$};
\node at (1.2,2.2) {$ab_1$};
\node at (3.5,4.5) {$ab_{s-1}$};
\node at (5.2,6.2) {$ab_s$};
\end{tikzpicture}
    \caption{The left and right basic chains}
    \label{fig_chains_ij}
\end{figure} 

We call the sequence $u_0,u_1,\ldots,u_r$ and the sequence $u'_0,u'_1,\ldots,u'_s$ in the statement of Proposition~\ref{chains_ij}, respectively, a \emph{left basic chain} and  a \emph{right basic chain}.

Proposition \ref{chains_ij} has the following notable consequence.
\begin{cor}\label{iMjM}
    There exists $z\in [u,v]$ such that $M(z)=(i(M),j(M))z$.
\end{cor}

\begin{proof}
Keep the notations of Proposition~\ref{chains_ij}. For $p\in [r]$ and $q\in [s]$, let $u_{p,0}=u_p$ and $u_{0,q}=u'_q$, and define recursively $u_{p,q}=(a_{p-1},a_{p})u_{p-1,q}=(b_{q-1},b_{q})u_{p,q-1}$, i.e. 
\[
u_{p,q}=[\di a_r\ldots a_{p+1}\di  a_{p-1}\cdots a_0 \di a_p \di b_q \di b_0 \ldots b_{q-1}\di b_{q+1}\ldots b_s\di ].
\]
Note $u_{p,q}\rhd u_{p-1,q}$ and $u_{p,q}\rhd u_{p,q-1}$. We have
\[
u_{r,s}=[\di  a_{r-1}\cdots a_0 \di a_r \di b_s \di b_0 \ldots b_{s-1}].
\]
The first $r+1$ rows of $T(u_{r,s})$ agree with those of $T(u_{r,0})$ while the last $s+1$ rows of $T(u_{r,s})$ agree with those of $T(u_{0,s})$. We deduce that $u_{r,s}\in [u,v]$, and hence $u_{p,q}\in [u,v]$ for all $p\in [r]$ and $q \in [s]$.

We claim that $M(u_{p,q})=(a_p,b_q)u_{p,q}$ for all $p\in [r]$ and $q \in [s]$, and we prove this by induction on $p+q$. If $p=0$ or $q=0$, then this follows from Proposition~\ref{chains_ij}. Suppose $p,q>0$, see Figure~\ref{FigCoriMjM}.
\begin{figure}[h]
\centering
    \begin{tikzpicture}
    \draw (-2,2) node[ thick,draw=black, inner sep=0,outer sep=0.5mm, minimum size=2mm, circle](y) {};
    \draw (2,2) node[ thick,draw=black, inner sep=0,outer sep=0.5mm, minimum size=2mm, circle](x) {};
    \draw (-2,4) node[ thick,draw=black, inner sep=0,outer sep=0.5mm, minimum size=2mm, circle](My) {};
    \draw (0,4) node[ thick,draw=black, inner sep=0,outer sep=0.5mm, minimum size=2mm, circle](x') {};
    \draw (2,4) node[ thick,draw=black, inner sep=0,outer sep=0.5mm, minimum size=2mm, circle](Mx) {};
    \draw (0,6) node[ thick,draw=black, inner sep=0,outer sep=0.5mm, minimum size=2mm, circle](Mx') {};
 \draw[line width=1mm, color=red] (x)--(Mx);
 \draw[line width=1mm, color=red] (y)--(My);
 \draw[line width=1mm, color=red] (x')--(Mx');
 \draw[thick] (x)--(x');
 \draw[thick] (y)--(x');
 \draw[thick] (Mx)--(Mx');
 \draw[thick] (My)--(Mx');
 \node at (2.7,1.8) {$u_{p-1,q}$};
 \node at (-2.7,1.8) {$u_{p,q-1}$};
 \node at (0.5,4.2) {$u_{p,q}$};
 \node at (2.7,3) {$a_{p-1}b_q$};
 \node at (-2.7,3) {$a_{p}b_{q-1}$};
 \node at (0.3,3.1) {$a_{p-1} a_p$};
 \node at (-0.7,2.5) {$b_{q-1}b_q$};
 
 \end{tikzpicture}
    \caption{Proof of Corollary \ref{iMjM}}
    \label{FigCoriMjM}
\end{figure}
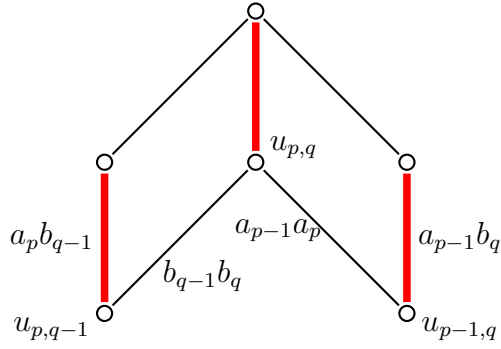
By Lemma~\ref{quadrati} applied to the interval $[u_{p,q-1},M(u_{p,q})]$, we have either $M(u_{p,q})=(a_p, b_{q-1}) u_{p,q}$ or $M(u_{p,q})=(a_p, b_{q}) u_{p,q}$.

By Lemma~\ref{quadrati} applied to the interval $[u_{p-1,q},M(u_{p,q})]$, we have either $M(u_{p,q})=(a_{p-1}, b_{q}) u_{p,q}$ or $M(u_{p,q})=(a_p, b_{q}) u_{p,q}$. 
Since $a_{p-1},a_p, b_{q-1}, b_q$ are distinct, the claim is proved.

The assertion follows since we can choose the element $z=u_{r,s}$.
\end{proof}

\section{Towards a contradiction I}
\label{badcasesI}

In this section, we begin the study of the structure of an interval admitting a special matching that does not fall into the desired classification. This will eventually lead to a contradiction.

Throughout this section, we fix a special matching $M$ of an interval $[u,v]$ in a Coxeter group $W$ of type $A$. We set $J=J(M)$ and $s=s(M)$. Suppose that $M(x)\neq \lambda_s^J(x)$ for some $x\in [u,v]$, and let $z$ be minimal with the property that
\[
M(z)\neq \lambda_s^J(z)
\qquad\text{and}\qquad
M(z)\lhd z.
\]
Recall that this means that there exists $h\in [n]$ such that, for all $x<z$ with $x\neq M(z)$, the matching $M$ swaps the digits occupying positions $h$ and $h+1$ among the red digits, i.e., the digits in $J$ (recall that $h$ satisfies $s=s_{h+i-1}$, where $i=\min J$).

The first result shows that $M$ is given by left multiplication by a fixed transposition on the whole subinterval $[u,z]$.

\begin{prop}
\label{abababab}
There exist $a,b$ in $[n]$ such that, for all $x$ in $[u,z]$, we have $M(x)= (a,b)x$.
\end{prop}
\begin{proof}
Let $a,b\in [n]$ be such that $M(z)= (a,b)z$. By Lemma~\ref{abovunque} (up-side down version), it is enough to show that $M(x)= (a,b)x$ holds for all $x$ in $[u,v]$ such that $x\lhd z$. Towards a contradiction, let $r$ in $[u,v]$ with $r\lhd z$ and $M(r)\neq (a,b)r$. By Lemma~\ref{quadrati} and the symmetry between $a$ and $b$, we may suppose $M(r)= (a,c) r$, $r= (c,b)z$, and $M(r)= (c,b)M(z)$. Hence, $c\in J$, i.e., $c$ is red, and $z$ is obtained from $r$ by swapping two red entries. Since $M(r)=\lambda_s^J(r)$, it follows $M(z)=\lambda_s^J(z)$, which is a contradiction.  
\end{proof}

For the rest of this section, let $a$ and $b$ denote the digits such that $M(x)= (a,b)x$ for all $x$ in $[u,z]$. We may suppose $a<b$.

\begin{lem}
\label{coatomi_z}
Let $w$ be a coatom of $[u,z]$, with $w\neq M(z)$. Suppose $w=(f,g)z$, with  $f,g \in [n] \setminus \{a,b\}$. Then
\begin{enumerate}
\item 
\label{coatomi_z_1}
$|\{f,g\}\cap J|=1$ (i.e., among $f$ and $g$, one is red and the other is black);
\item 
\label{coatomi_z_2}
in the one-line notation of $z$, the digits $f$ and $g$ cannot both appear before $a$ and $b$, nor both between $a$ and $b$, nor both after $a$ and $b$;
\item 
\label{riassunto}
one of the following holds:
\[
a<b<\min\{f,g\}<\max\{f,g\}
\]
or
\[
\min\{f,g\}<\max\{f,g\}<a<b.
\]
\end{enumerate}
\end{lem}
\begin{proof}
We may suppose $f<g$. \\
Let us prove (\ref{coatomi_z_1}) and (\ref{coatomi_z_2}).
The one-line notation of $w$ has exactly $h-1$ red digits to the left of $b$, and $a$ is the first red digit to the right of $b$. This is still true after swapping two elements having the same color, and  (\ref{coatomi_z_1}) follows. This is also true after swapping  two elements both appearing before $a$ and $b$, both appearing between $a$ and $b$, or both appearing after $a$ and $b$,  so (\ref{coatomi_z_2}) follows.\\
Let us prove (\ref{riassunto}). 
If $a<f<g<b$ holds, then $f$ and $g$ are red by definition, which contradicts (\ref{coatomi_z_1}).
If $f<a<b<g$, then, since $(f,g)z\lhd z$, in the one-line notation of $z$ the digit $g$ is to the left of $f$, and neither $a$ nor $b$ are between $f$ and $g$. This contradicts (\ref{coatomi_z_2}).
If $a<f<b<g$, then, since $(f,g)z\lhd z$, in the one-line notation of $z$ the digit $g$ is to the left of $f$, and $b$ is not between $g$ and $f$. Since $(a,b)z\lhd z$, the digit $f$ is not between $b$ and $a$. Furthermore, $f$ to the left of $b$ contradicts (\ref{coatomi_z_2}). Suppose that $f$ is to the right of $a$. Since $b$ is not between $g$ and $f$,  the digit  $g$ is not to the left of $b$. If $g$ is between $b$ and $a$, then $f$ is between $b$ and $a$ in the one-line notation of $w$, which is impossible since $M(w)=\lambda_s^J(w)$ and $f$ is red  being $a<f<b$. The only remaining case, i.e., $g$ to the right of $a$, contradicts (\ref{coatomi_z_2}).
If $f<a<g<b$, the companion of $M$ falls in the previous case. 

The only remaining possibilities are those listed in the statement.
\end{proof}

\begin{lem}
\label{coatomi_z_a}
Let $w$ be a coatom of $[u,z]$, with $w\neq M(z)$. Suppose $w=(f,g)z$, with  $f,g \in [n]$. 
\begin{enumerate}
\item 
\label{coatomi_z_a1}
Let $g=b$. Then 
\begin{itemize}
    \item $f<a$, 
    \item the digit $f$ appears in $z$ between $b$ and $a$,
    \item  there is at least one red digit that appears in $z$ between $b$ and $f$ with $f$ included.
    \end{itemize}
\item 
\label{coatomi_z_a2}
Let $f= a$. Then  
\begin{itemize}
    \item $g>b$, 
    \item  the digit $g$ appears in $z$ between $b$ and $a$, 
    \item  there is at least one red digit that appears in $z$ between $a$ and $g$ with $g$ included.
    \end{itemize}
\end{enumerate}
\end{lem}
\begin{proof}
The second statement is equivalent to the first applied to the companion matching. 

Let us prove (\ref{coatomi_z_a1}). The digit $f$ is not to the right  of $a$ in $z$ because otherwise $(a,b)w\rhd w$ would hold. Towards a contradiction, suppose that $f$ is  to the left  of $b$ in $z$. 
Then $f$ and all digits between $f$ and $b$ must be black; otherwise, there would be red digits between $b$ and $a$ in $w$, which is a contradiction since $M(w)=\lambda_s^J(w)$. Then, in  $z$, the digits $b$ and $a$ are, as in $w$, the $h$-th and $(h+1)$-st red digits. This contradicts $M(z)\neq \lambda_s^J(z)$. 
The only remaining possibility is that $f$ appears between $b$ and $a$ in $z$.

Let us show $f<a$. Since $M(z)=(a,b)z\lhd z$ and $M(w)=(a,b)w\lhd w$, we have $M(w)= (a,f)M(z)$, by Lemma~\ref{quadrati}. Since $(a,f)M(z)\lhd M(z)$, it follows $f<a$.

Since  $M(w)=\lambda_s^J(w)$ whereas $M(z)\neq \lambda_s^J(z)$, at least one digit between $b$ and $f$ with $f$ included must be red.
\end{proof}

\begin{prop}
\label{coatomi_di_z}
Let $w$ be a coatom of $[u,z]$, with $w\neq M(z)$. Then exactly one of the following holds:
\begin{enumerate}
\item[(c1)] $w=(c,d)z$, with $a<b<c<d$, the digit $c$ is red, $d$ is black, and  $z=[\di d \di b \di a \di c \di ]$;
\item[(c2)] $w=(c,d)z$, with $a<b<c<d$, the digit $c$ is red, $d$ is black, and  $z=[\di d \di b \di c \di a \di]$;
\item[(c3)] $w=(b,p)z$, with $p<a<b$, $z=[\di  b \di p \di a \di ]$ and  at least one digit appearing between $b$ and $p$, with $p$ included, is red;

\item[(c1$^*$)] $w=(p,q)z$, with $p<q<a<b$, the digit $q$ is red, $p$ is black, and  $z=[\di q \di b \di a \di p \di ]$;
\item[(c2$^*$)] $w=(p,q)z$, with $p<q<a<b$, the digit $q$ is red, $p$ is black, and $z=[\di b \di q \di a \di p \di ]$;
\item[(c3$^*$)] $w=(a,c)z$, with $a<b<c$, $z=[\di  b \di c \di a \di ]$ and at least one digit appearing between $c$ and $a$, with $c$ included, is red.

\end{enumerate}
\end{prop}
\begin{proof}
We let $w=(f,g)z$, with $f<g$.

Suppose $f,g \in [n] \setminus \{a,b\}$. By Lemma~\ref{coatomi_z}, (\ref{riassunto}), we have $a<b<f<g$ or $f<g<a<b$.
Suppose first $a<b<f<g$. By Lemma~\ref{coatomi_z},  (\ref{coatomi_z_1}), the digit $f$ is red and $g$ is black. By Lemma~\ref{coatomi_z},  (\ref{coatomi_z_2}), one of the following holds:
\begin{enumerate}
\item
$z=[\di g \di b \di a \di f \di ]$, or 
\item $z=[\di g \di b \di f \di a \di ]$, or 
\item $z=[\di b \di g \di a \di f \di ]$.
\end{enumerate}
These are the only possibilities, since $(a,b)z \lhd z$ and  $(f,g)z \lhd z$. In the first case, we get (c1). In the second case, we get (c2). The third case is impossible; the one-line notation of $w$ would have $f$ between $b$ and $a$, which contradicts $M(w)=\lambda_s^J(w)$ since $f$ is red. If $f<g<a<b$, then we can consider the companion matching of $M$ and we get (c1$^*$) and (c2$^*$).

By Lemma~\ref{coatomi_z_a},  if $f=b$ we get (c3), and if $g=a$ we get  (c3$^*$).
\end{proof}

We now study the compatibility of the different types of coatoms. 

\begin{lem}
\label{com_q}
Let $w_1$ and $w_2$ be two coatoms of $[u,z]$ such that there exists $t$ in $[u,z]$ with  $t \lhd w_i$, for $i\in \{1,2\}$. Let $w_i = (f_i,g_i)z$, for $i\in \{1,2\}$. Then $|\{f_1,g_1,f_2,g_2\} \setminus \{a,b\}|<4$. 
\end{lem}
\begin{proof}
Towards a contradiction, suppose $|\{f_1,g_1,f_2,g_2\} \setminus \{a,b\}|=4$. Consider the one-line notation of $t$, where $b$ and $a$ are, respectively, the $h$-th and $(h+1)$-st red digits. 
This is still true after one swap of $f_i$ and $g_i$, for some $i\in \{1,2\}$. Hence, the same property holds after performing the two swaps in succession, and therefore for the one-line notation of $z$, yielding a contradiction.  
\end{proof}

\begin{lem}
\label{quali_q}
Let $w_1$ and $w_2$ be two coatoms of $[u,z]$. The following holds.
\begin{enumerate}
\item
\label{quali_q1}
Let $w_1$ be of type \textnormal{(c1)}. Then $w_2$ is neither of type \textnormal{(c2)}, nor \textnormal{(c3)}, nor \textnormal{(c3$^*$)}. 
\item
\label{quali_q1*}
Let $w_1$ be of type \textnormal{(c1$^*$)}. Then $w_2$ is neither of type \textnormal{(c3)}, nor \textnormal{(c2$^*$)}, nor \textnormal{(c3$^*$)}.   
\item
\label{quali_q2}
Let $w_1$ be of type \textnormal{(c2)}. Then $w_2$ is neither of type \textnormal{(c1)}, nor of type \textnormal{(c3$^*$)}. Moreover, when both $w_1$ and $w_2$ are of type \textnormal{(c2)}, then  $w_1 = (c_1,d_1)z$ and $w_2 = (c_2,d_2)z$, with $c_1=c_2$.
\item
\label{quali_q2*}
Let $w_1$ be of type \textnormal{(c2$^*$)}. Then $w_2$ is neither of type \textnormal{(c3)}, nor of type \textnormal{(c1$^*$)}. Moreover, when both $w_1$ and $w_2$ are of type \textnormal{(c2$^*$)}, then  $w_1 = (p_1,q_1)z$ and $w_2 = (p_2,q_2)z$, with $q_1=q_2$.
\item
\label{quali_q3}
Let $w_1$ be of type \textnormal{(c3)}. Then  $w_2$ is neither of type \textnormal{(c1)}, nor \textnormal{(c1$^*$)}, nor \textnormal{(c2$^*$)}, nor \textnormal{(c3$^*$)}. 
\item
\label{quali_q3*}
Let $w_1$ be of type \textnormal{(c3$^*$)}. Then  $w_2$ is neither of type \textnormal{(c1)}, nor \textnormal{(c2)}, nor \textnormal{(c3)}, nor \textnormal{(c1$^*$)}. 
\end{enumerate}
\end{lem}
\begin{proof}
(\ref{quali_q1}).
Let $w_1 = (c_1,d_1)z$, with $c_1$ red and $d_1$ black. Then $z=[\di d_1 \di b \di a \di c_1 \di]$.  

Suppose that $w_2$ is of type (c2).  Let  $w_2 = (c_2,d_2)z$, with $c_2$ red and $d_2$ black. Then  $z=[\di d_i \di d_j \di b \di c_2 \di a \di c_1 \di ]$ with $\{i,j\}=\{1,2\}$. Then $c_2$ is between $b$ and $a$ in the one-line notation of $w_1$, which contradicts $M(w_1)=\lambda_s^J(w_1)$.

Suppose that $w_2$ is of type (c3). Let  $w_2 = (b,p)z$. Then 
 $z=[\di d_1  \di b \di p \di a \di c_1 \di ]$ and at least one digit between $b$ and $p$ with $p$ included is red. Then there is a red digit between $b$ and $a$ in  $w_1$, which is a contradiction.

Suppose  that $w_2$ is of type (c3$^*$). Let  $w_2 = (a,c)z$. Then 
 $z=[\di d_1  \di b \di c \di a \di c_1 \di ]$ and at least one digit between $c$ and $a$ with $c$ included is red. Then there is a red digit between $b$ and $a$ in $w_1$, which is a contradiction. 

(\ref{quali_q1*}). It is obtained from (\ref{quali_q1}) by considering the companion matching of $M$.

(\ref{quali_q2}).
Let $w_1 = (c_1,d_1)z$, with $c_1$ red and $d_1$ black. Then $z=[\di d_1 \di b \di c_1 \di a \di ]$.  

By (\ref{quali_q1}), the type of $w_2$ is not (c1).

Suppose  that $w_2$ is of type (c3$^*$). Let  $w_2 = (a,c)z$. Then either $z=[\di d_1  \di b \di c \di c_1 \di a  \di ]$ or
 $z=[\di d_1  \di b \di c_1 \di c \di a  \di ]$ and, in both cases, at least one digit between $c$ and $a$ with $c$ included is red. Then there is a red digit between $b$ and $a$ in $w_1$, which is a contradiction. 

Let  $w_2$ be of type (c2). Let $w_2 = (c_2,d_2)z$, with $c_2$ red and $d_2$ black. Towards a contradiction, suppose $c_1\neq c_2$. Then  $z=[\di d_{i} \di d_{j} \di b \di c_{i'} \di c_{j'} \di a \di]$,  with $\{i,j\}=\{i',j'\}=\{1,2\}$. Then $c_1$ (respectively, $c_2$) appears between $b$ and $a$ in $w_2$ (respectively, $w_1$), which is a contradiction.

(\ref{quali_q2*}). It is obtained from (\ref{quali_q2}) by considering the companion matching of $M$.

(\ref{quali_q3}).
Let $w_1 = (b,p)z$. Then $z= [\di b \di p \di a \di ]$.  By (\ref{quali_q1}), (\ref{quali_q1*}), (\ref{quali_q2*}), we need to show that the type of  $w_2$ is not (c3$^*$).

Towards a contradiction, suppose that $w_2$ is of type (c3$^*$). Let $w_2= (a,c)z$. 
In both $w_1$ and $w_2$,
the digits $b$ and $a$ are the $h$-th and $(h+1)$-st red digits. Hence in $(b,p)w_1$, the digit $a$ is the $(h+1)$-st red digit, and in $(a,c)w_2$ the digit $b$ is the $h$-th red digit. But both $(b,p)w_1$ and $(a,c)w_2$ are $z$, which contradicts $M(z)\neq \lambda_s^J(z)$. 

(\ref{quali_q3*}). It is obtained from (\ref{quali_q3}) by considering the companion matching of $M$.
\end{proof}

The following lemma holds for every interval in every Coxeter group.
\begin{prop}
\label{catena_di_quadrati}
Let $[x,y]$ be a Bruhat interval in an arbitrary Coxeter group. Let $w_1$ and $w_2$ be two coatoms of $[x,y]$. Then there exist $r \in \mathbb N$ and a sequence of coatoms $c_0,c_1,\ldots,c_r$ such that $c_0=w_1$, $c_r=w_2$, and, for each $i\in [r]$, there exists $t_i$ in $[x,y]$ such that $t_i \lhd c_{i-1}$ and $t_i \lhd c_{i}$. 
\end{prop}
\begin{proof}
The result follows by \cite[Proposition~2.1]{CDM}, which states that, given  two complete chains $C$ and $C'$ from $x$ to $y$, 
there exist $k \in \mathbb N$ and a sequence $C =C_1, C_2, \ldots , C_k= C'$ of complete chains from $x$ to $y$ such that $C_i$ and $C_{i+1}$ differ by one element, for all $i \in [1,k-1]$.
\end{proof}

\begin{prop}
\label{class_finale}
One of the following mutually exclusive statements holds.
\begin{enumerate}
\item \label{class_finale_1}
Every coatom of $[u,z]$ distinct from $M(z)$ is of type  (c1).
\item \label{class_finale_2} Every coatom of $[u,z]$ distinct from $M(z)$ is of type (c2) or of type (c3), and when two coatoms $w_1$ and $w_2$ are both of type (c2), they satisfy  $w_1 = (c_1,d_1)z$ and $w_2 = (c_2,d_2)z$, with $c_1=c_2$.
\item \label{class_finale_1*} Every coatom of $[u,z]$ distinct from $M(z)$ is of type  (c1$^*$).
\item \label{class_finale_2*} Every coatom of $[u,z]$ distinct from $M(z)$ is either of type  (c2$^*$) or of type (c3$^*$), and when two coatoms $w_1$ and $w_2$ are both of type (c2$^*$), they satisfy  $w_1 = (p_1,q_1)z$ and $w_2 = (p_2,q_2)z$, with $q_1=q_2$.
\end{enumerate}
\end{prop}
\begin{proof}
Let $w$ be a coatom of $[u,z]$, with $w\neq M(z)$.

Suppose first that $w$ is of type (c1) and let $w'$ be another coatom. By Lemma~\ref{quali_q}, (\ref{quali_q1}), the type of $w'$ is neither (c2), nor (c3), nor (c3$^*$). Towards a contradiction, suppose that $w'$ is of type either (c1$^*$) or (c2$^*$). Then, along the sequence whose existence is guaranteed by Proposition~\ref{catena_di_quadrati}, there is an interval of length 2 with $z$ as the top element and containing two coatoms $r_1$ and $r_2$, with $r_1$ of type (c1) and $r_2$ of type (c1$^*$) or (c2$^*$). This contradicts  Lemma~\ref{com_q}.
Hence, (\ref{class_finale_1}) holds. 

Suppose now that $w$ is of type (c2) and let $w'$ be another coatom. By Lemma~\ref{quali_q}, (\ref{quali_q2}), we need to show that $w'$ can be neither of type (c1$^*$), nor of type (c2$^*$). Towards a contradiction, suppose that $w'$ is of type either (c1$^*$) or (c2$^*$). Then, as above, we notice that, along the sequence whose existence is guaranteed by Proposition~\ref{catena_di_quadrati}, there is an interval of length 2 with $z$ as the top element and containing two coatoms $r_1$ and $r_2$, with $r_1$ of type (c2) and $r_2$ of type (c1$^*$) or (c2$^*$). This contradicts  Lemma~\ref{com_q}.
Hence, (\ref{class_finale_2}) holds. 

Suppose now that $w$ is of type (c3). Then (\ref{class_finale_2}) holds by Lemma~\ref{quali_q}), (\ref{quali_q3}).

When $w$ is of type (c1$^*$), (c2$^*$), or (c3$^*$),  the result follows by the previous discussion applied to the companion matching.  
\end{proof}

\section{Towards a contradiction II} \label{badcases}

As in the previuos section, we fix a special matching $M$ of an interval $[u,v]$ in a Coxeter group $W$ of type $A$. We set $J=J(M)$ and $s=s(M)$ and suppose that $M(x)\neq \lambda_s^J(x)$ for some $x\in [u,v]$. Let $z$ be minimal with the property that
\[
M(z)\neq \lambda_s^J(z)
\qquad\text{and}\qquad
M(z)\lhd z.
\]

In this section, we show that actually none of the mutually exclusive conditions in Proposition~\ref{class_finale} can occur, finding a contradiction and thus proving that all special matchings are of the form $\lambda_s^J$.

We will use the following well-known result several times without explicit mention.
\begin{lem}
Let $x,y \in W$. The group generated by $\{cy^{-1} :  c\in [x,y], c\lhd y\}$ coincides with the group generated by $\{pq^{-1} : p,q \in [x,y], p\lhd q\}$.
\end{lem}

\subsection{Coatoms of type (c1)}
In this subsection, we suppose that all coatoms of $z$ distinct from $M(z)$ are of type (c1). In particular, we have
\begin{equation}\label{zc1}
    z=[\di d_1 \ldots d_r \di \color{red}b \color{black}\di \color{red}a \color{black}\di \color{red}c_1 \color{black}\ldots \color{red}c_s \color{black}\di],
\end{equation}
for some $r,s\geq 1$, with $a<b<c_1<\cdots<c_s<d_1<\cdots < d_r$ where $c_p$ is red for all $p\in [s]$ and $d_q$ is black for all $q\in [r]$, with the further condition that for all $p$ there exists $q$ such that $(c_p, d_q)z$ is a coatom and, similarly, for all $q$ there exists $p$ such that  $(c_p, d_q)z$ is a coatom. We observe that all elements in $[u,z]$ are obtained by a permutation of the entries $c_n$ and $d_m$, and possibly swapping $a$ and $b$. Moreover, since $M(x)=(a,b)x=\lambda_s^J(x)$ for all $x$ in $[u,z]\setminus \{z,M(z)\}$, all such elements have exactly one element $c_p$ in the positions occupied by $d_1,\ldots,d_r$ in $z$ (they must have the same number of red digits to the left of $a$).

\begin{lem}
\label{uinc1}
Being 
    $z=[\di d_1 \ldots d_r \di \color{red}b \color{black}\di \color{red}a \color{black}\di \color{red}c_1 \color{black}\ldots \color{red}c_s \color{black}\di]$,
we have \[u= [\di \color{red}c_1 \color{black}\di d_1 \ldots d_{r-1}\di \color{red}a \color{black}\di \color{red}b \color{black}\di \color{red}c_2 \color{black}\ldots \color{red}c_s \color{black}\di d_r \di ].\]
\end{lem}
\begin{proof}
Let $k$ be the position of $d_1$ in $z$, i.e., $k= z^{-1}(d_1)$, and note by the previous discussion that $u(k)\in \{c_1,\ldots,c_s,d_1,\ldots,d_r\}$. Since there is a coatom $w_1$ such that $w_1(k)=c_p$ for some $p$, necessarily $u(k)=c_l$ for some $l$, otherwise $u\nleq w_1$ by the Tableau Criterion. Moreover, since there exists a coatom $w_2=(c_1,d_q)z$ for some $q$, necessarily $l=1$, otherwise $u\nleq w_2$ by the Tableau Criterion.
Considering a coatom $w_3=(d_2,c_p)z$ for some $p$, we deduce $u= [\di c_1 \di d_1 \di ]$ with $u^{-1}(d_1)=z^{-1}(d_2)$, and an inductive argument shows $u= [\di c_1 \di d_1\ldots d_{r-1} \di]$, with $u^{-1}(d_{m})=z^{-1}(d_{m+1})$ for all $m=1,\ldots,r-1$. 
This argument applied to the companion matching completes the proof.
\end{proof}
\begin{prop}\label{prop(c1)}
It is not possible that all coatoms of $[u,z]$ distinct from $M(z)$ are of type \textnormal{(c1)}.
\end{prop}

\begin{proof}Omitting (here and elsewhere) the symbols $\di$ for the sake of clarity of notation, we have
\[
z=[d_1d_2\ldots d_r \color{red}b a c_1c_2\color{black}\ldots \color{red}c_s\color{black}]
\]
and 
\[
u=[\color{red}c_1 \color{black}d_1\ldots d_{r-1} \color{red}a b c_2\cdots c_s \color{black}d_r]
\]
by Lemma \ref{uinc1}.
We first consider the case $s=1$, i.e.
\[
u=[\color{red}c  \color{black}d_1\ldots d_{r-1} \color{red} a  b  \color{black}d_r ]
\]
and
\[
z= [d_1 d_2\ldots d_r\color{red} b a c\color{black}]
\]
where $a<b<c<d_1<\cdots<d_r$ with $c$ red and $d_q$ black for all $q$. Since $c$ is red, by Proposition \ref{chains_ij} there exists a right basic chain given by $b_1,\ldots, b_m$ (all red) with $b_1<\cdots <b_{m-1}<c<b_m$ such that 
\[
u=[\color{red}c  \color{black}d_1  \ldots  d_{r-1}  \color{red}a  b  b_1 \ldots  b_{k} \color{black}d_r  \color{red}b_{k+1} \ldots b_m \color{black}]
\]
and so
\[
z=[d_1d_2 \ldots d_r \color{red}b a b_1\color{black}\ldots \color{red}b_kcb_{k+1} \ldots b_m\color{black}]
\]
where we observe that $k< m$ (i.e., $b_m$ is to the right of $d_r$ in $u$) because $(c,d_1)z\lhd z$.
For notational convenience, let $b_0=b$. Lemma \ref{chains_ij} provides that the element
\[
u'_m=[\color{red}c \color{black}d_1\ldots d_{r-1} \color{red}a  b_m b_0\color{black}\ldots \color{red}b_{k-1} \color{black}d_r \color{red}b_k \cdots b_{m-1}\color{black}]
\]
satisfies $M(u'_m)=(a,b_m)u'_m$.
Now consider the element
\[
\alpha =[d_1 d_2\ldots d_r \color{red}a c b_0 \color{black}\ldots \color{red}b_{k-1}b_m b_{k}\color{black}\ldots \color{red}b_{m-1}\color{black}].
\]
The restricted tableaux have $r+m+3$ rows. We observe that $T(\alpha)$ agrees with $T(z)$ in the first $r$ rows and with $T(M(u'_m))$ in the last $m-k$ rows. The other rows are shown in the following table:

\begin{center}\begin{tabular}{|l|l|l|}
\hline
$T(\alpha)$    &$T(z)$&$T(M(u'_n))$  \\
\hline \hline
    $ad_1\ldots d_r$ &$b_0d_1\ldots d_r$ & $c b_m d_1\ldots d_{r-1}$ \\
    \hline
    $acd_1\ldots d_r$ &$ab_0d_1\ldots d_r$ & $ac b_m d_1\ldots d_{r-1}$ \\
    \hline
    $ab_0cd_1\ldots d_r$ &$ab_0b_1d_1\ldots d_r$ & $ab_0c b_m d_1\ldots d_{r-1}$ \\
    \hline
    $\ldots$&$\ldots$ & $\ldots$\\
    \hline
    $ab_0\ldots b_{k-1}cd_1\ldots d_r$ &$ab_0\ldots b_kd_1\ldots d_r$ & $ab_0\ldots b_{k-1}c b_m d_1\ldots d_{r-1}$ \\
    \hline
    $ab_0\ldots b_{k-1}cb_m d_1\ldots d_r$ &$ab_0\ldots b_kcd_1\ldots d_r$ & $ab_0\ldots b_{k-1}c b_m d_1\ldots d_{r}$ \\
    \hline
\end{tabular}\end{center}

and so we can deduce that $T(\alpha)$ is smaller than the maximum of the tableaux of $z$ and $M(u_m')$, and in particular $\alpha \leq v$.
We also observe that $\alpha$ is obtained from $u_m'$ by a sequence of multiplications on the left by transpositions that involve elements $d_i$, namely:
\[\alpha=(c,d_r)(c,d_{r-1})\cdots (c,d_1)(b_m,d_r)u'_m,
\]
which increases in Bruhat order at every step.
This implies that $\alpha\geq u_m'$ and in particular $\alpha \in [u,v]$. 

Lemma~\ref{absednero} implies $M(\alpha)=(a,b_m)\alpha$, but this is a contradiction since $(a,b_m)\alpha$ does not cover $\alpha$, because $a<c<b_m$.

Now we tackle the case $s>1$. Since $c_2$ is red, there exists a right basic chain, i.e. $m\geq 1$ and $b_1,\ldots,b_m$ (all red) with $b_{m-1}<c_2\leq b_m$,
with 
\[
u_k'=(b_{k-1},b_k)u_{k-1}'\rhd u_{k-1}'
\]
and $M(u_k')=(a,b_k)u_k'$ for all $k\in [m]$.
Assume $b_m>c_2$. We observe that $b_m$ cannot sit after $c_2$ in $u$ otherwise we cannot have the covering relation  $u_{m-1}'\lhd 
u_m'$. But if $b_m$ sits before $c_2$ in $u$ then it would sit before $c_2$ also in $z$, and so 
$(c_2,d_q)z \not \!\! \lhd z$ for all $q$, which is a contradiction.

We deduce  $c_2=b_m$ and so 
\[
u=[\color{red}c_1 \color{black}d_1 \ldots d_{r-1} \color{red} a b_0 \ldots b_{m-1} c_2 \color{black}\ldots \color{red}c_s \color{black}d_r]
\]
and
\[
z=[d_1d_2\ldots d_r \color{red}b_0 a b_1 \color{black}\ldots \color{red}b_{m-1} c_1 \ldots c_{s-1}c_s \color{black}]
\]
We observe $b_1,\ldots,b_{m-1}<c_1$ otherwise $(c_1,d_q)z \not \!\!\lhd z$ for all $q\in [r]$  and so we have
\[
a<b_0<\cdots<b_{m-1}<c_1<b_m=c_2<\cdots c_s<d_1<\cdots d_r.
\]
By construction and Lemma \ref{chains_ij} (recall that $b_m=c_2)$, we have
\[u'_m= [\color{red}c_1 \color{black}d_1\ldots d_{r-1}\color{red}a c_2 b_0
b_1  \color{black}\ldots \color{red}b_{m-1}c_3 \ldots \color{red}c_s  \color{black}d_r]
\]
and recall that  $M(u'_m)=(a,c_2)u'_m$.
Consider now the element 
\[
\alpha=[\color{red}c_1  \color{black}d_1\ldots d_{r-1} \color{red}a  \color{black}d_r \color{red}b_0 \color{black}\ldots \color{red}b_{m-1}c_2  \color{black}\ldots \color{red}c_s \color{black}].
\]
Note that $T(\alpha)$ agrees with $T(u'_m)$ in the first $r$ rows and with $T(z)$ in the last $s-1$ rows. The remaining rows are shown in the following table
\begin{center}

\begin{tabular}{|l|l|l|}
\hline
$T(\alpha)$    &$T(z)$&$T(M(u'_m))$  \\
\hline \hline
$ac_1d_1\ldots d_{r-1}$ & $b_0 d_1\ldots d_r$& $c_1 c_2 d_1\ldots d_{r-1}$\\ 
\hline
$ac_1d_1\ldots d_r$ & $ab_0 d_1\ldots d_r$& $ac_1 c_2 d_1\ldots d_{r-1}$\\ 
\hline
$ab_0c_1d_1\ldots d_r$ & $ab_0 b_1 d_1\ldots d_r$& $ab_0c_1 c_2 d_1\ldots d_{r-1}$\\ 
\hline
$\ldots$ & $\ldots$ &$\ldots$ \\
\hline
$ab_0\ldots b_{m-2}c_1d_1\ldots d_r$ & $ab_0 \ldots b_{m-1} d_1\ldots d_r$& $ab_0\ldots b_{m-2} c_1 c_2 d_1\ldots d_{r-1}$\\ 
\hline
$ab_0\ldots b_{m-1}c_1d_1\ldots d_r$ & $ab_0 \ldots b_{m-1} c_1d_1\ldots d_r$& $ab_0\ldots b_{m-1} c_1 c_2 d_1\ldots d_{r-1}$\\ 
\hline
\end{tabular}
\end{center}

We can deduce that $\alpha \leq v$.
Note also that $\alpha=(c_2,d_r) \cdots (c_s,d_r)u_m'$ and we deduce that $\alpha\in [u,v]$ and by Lemma~\ref{absednero}, $M(\alpha)=(a,c_2)\alpha$, which is a contradiction since $a<b_0<c_2$.

\end{proof}
\subsection{Coatoms of type (c2$^*$)}
In this subsection, we suppose that all atoms of $[u,z]$ distinct from $M(z)$ are of type (c2$^*$). By Lemma \ref{quali_q}, there exists $r>0$ such that
\[z=[\color{red}b q a \color{black}p_1\ldots p_r]
\]
with $p_1<\cdots<p_r <q<a<b$ where $p_1,\ldots,p_r$ are black and $q$ is red.
Moreover, the coatoms of $[u,z]$ are $M(z)$ and $w_l=(p_l,q)z$ for $l\in [r]. $
\begin{lem}
Being $z=[\color{red}b q a \color{black}p_1\ldots p_r]$, we have 
\[u=[\color{red}a \color{black}p_1\color{red}b\color{black}p_2\ldots p_r \color{red}q\color{black}].\] 
\end{lem}
\begin{proof}
  By the description of the coatoms, along the interval $[u,z]$, the entries  $$p_1,\ldots,p_r,q$$ are permuted among themselves, and possibly $a$ is swapped with $b$,  while all other entries remain in the same positions as in $z$. It follows that $u^{-1}(a)=z^{-1}(b)$ and $u^{-1}(b)=z^{-1}(a)$. Since $u\leq w_1$, we have $u^{-1}(p_1)=z^{-1}(q)$ and one can prove with a simple inductive argument that $u^{-1}(p_l)=z^{-1}(p_{l-1})$ for all $l\in [2,r]$, since $u\leq w_{l-1}$.   
\end{proof}
\begin{prop}
\label{prop(c2)}
It is not possible that all coatoms of $[u,z]$ distinct from $M(z)$ are of type \textnormal{(c2$^*$)}.
\end{prop}
\begin{proof}
   Since $q$ is red, there exist $a_1,\ldots,a_m$ (all red) with $a_m<q<a_{m-1}<\cdots<a_1<a_0=a$ as in Proposition~\ref{chains_ij}. In particular, we have
\[
u=[\color{red}a_m\color{black}\ldots \color{red}a_1a_0 \color{black}p_1 \color{red}b \color{black}p_2\ldots p_r \color{red}q\color{black}]
\]
and 
\[
u_m=[\color{red}a_{m-1}\color{black}\ldots \color{red}a_0 a_m \color{black}p_1 \color{red}b \color{black}p_2\ldots p_r \color{red}q\color{black}]
\]
with $M(u_m)=(a_m,b)u_m$.
Consider the element
\[
\alpha=(p_1,q)(p_2,q)\cdots (p_r,q)u_m.
\]
Then $\alpha>u_m$ and 
\[
\alpha=[\color{red}a_{m-1}\color{black}\ldots \color{red}a_0 a_m q b \color{black}p_1\ldots p_r].
\]
Note that $T(\alpha)$ agrees with $T(M(u_m))$ in the first $m$ rows and with $T(z)$ in the last $r$ rows. The remaining rows of $T(\alpha)$ and $T(z)$ are shown in the following table:

\begin{center}   
\begin{tabular}{|l|l|}
\hline
$T(\alpha)$    &$T(z)$  \\
\hline \hline
$a_m\ldots a_0$& $a_m\ldots a_1 b$ \\
\hline 
$a_mqa_{m-1}\ldots a_0$& $a_mqa_{m-1}\ldots a_1 b$ \\
\hline 
$a_mq a_{m-1}\ldots a_0 b$& $a_mq a_{m-1}\ldots a_0 b$ \\
\hline 
\end{tabular}
\end{center}

We deduce $\alpha \in [u,v]$. By Lemma~\ref{absednero}, we have $M(\alpha)=(a_m,b) \alpha$, which is a contradiction since $a_m<q<b$ and hence $(a_m,b) \alpha \not\rhd \alpha$. 
\end{proof}

\subsection{Coatoms of type (c3$^*$)}

In this subsection, we suppose that all atoms of $[u,z]$ distinct from $M(z)$ are of type (c3$^*$). 
Hence, there exists $s>0$ such that
\[
z=[\color{red} b\color{black}c_1\ldots c_s \color{red}a\color{black}]
\]
with $a<b<c_1<\cdots < c_s$, and there is at least one red entry between $c_l$ and $a$, $c_l$ included, for all $l\in [s]$. Moreover, the coatoms of $[u,v]$ are  $M(z)$ and $w_l=(a,c_l)z$ for $l\in [s]$.

\begin{lem}
Being $z=[\color{red} b\color{black}c_1\ldots c_s \color{red}a\color{black}]$, we have 
\[u=[\color{red}ab \color{black}c_1\ldots c_s].\]
\end{lem}
\begin{proof}
By the description of the coatoms, along the interval $[u,z]$ the entries
\[
a,b,c_1,\ldots,c_s
\]
are permuted among themselves, while all other entries remain in the same positions as in $z$.
Since $u\leq M(z)$, we have $u^{-1}(a)=z^{-1}(b)$. The condition $u\leq w_1$ then implies $u^{-1}(b)=z^{-1}(c_1)$, and then one can recursively verify that $u^{-1}(c_l)=z^{-1}(c_{l+1})$ for all $l\in [s-1]$ since $u\leq w_{l+1}$.
\end{proof}
\begin{prop}\label{prop(c3)}
It is not possible that all coatoms of $[u,z]$ distinct from $M(z)$ are of type \textnormal{(c3$^*$)}.
\end{prop}
\begin{proof}
Recall
\[
u=[\color{red}ab\color{black}c_1\ldots c_{s-1}c_s]
\]
and 
\[
z=[\color{red}b\color{black}c_1c_2\ldots c_s \color{red}a \color{black}].
\]

Consider the following claim: between $c_1$ and $a$ in $z$, with $c_1$ included, there are no red digits. If the claim holds, then the result also holds. Indeed, since there are no red digits between $a$ and $b$ in $u$, there are no red digits between $b$ and $c_1$ in $z$. The claim implies that there are no red digits between $a$ and $b$ in $z$. It would follow $M(z)=\lambda_s^J$, which is a contradiction. 

Let us prove the claim.
We first show that between $c_1$ and $a$ in $z$ there are no red digits greater than $a$. 
Towards a contradiction, suppose that there are such red digits. Clearly, such digits cannot be smaller than $b$. Hence, by Proposition~\ref{chains_ij}, there exists a right basic chain given by $b_1, \ldots, b_s$. In particular, $b_1$  sits to the right of $b$ in $u$, and satisfies $b_1>b$, as well as $u_1'=(b,b_1)u\rhd u$, and $M(u_1')=(a,b_1)u_1'$.

\color{black}

 Suppose that $b_1$ sits between $b$ and $c_s$ in $u$. Observing that $b_1\neq c_1$ (otherwise $u_1'\leq z$ and hence $M(u_1')=(a,b)z$), we deduce that $b_1$ sits between $c_1$ and $a$ in $z$ so $b_1>c_1$; thus, $b_1$ sits between $b$ and $c_1$ in $u$, because $u_1'\rhd u$. Therefore, we have
\[
u=[\color{red}abb_1 \color{black} c_1\ldots c_{s-1} c_s],
\]
\[
u_1'=[\color{red}ab_1b \color{black} c_1\ldots c_{s-1} c_s],
\]
and
\[
z=[\color{red}b\color{black}c_1\color{red}b_1\color{black}c_2\ldots c_s \color{red}a\color{black}]
\]
Consider the element 
$\beta=[\color{red}b b_1 \color{black}c_1\ldots c_s \color{red}a\color{black}]$. The restricted tableau of $\beta$ agrees with the tableau of $z$ in the last $s$ rows. Comparing the first three rows of the tableaux of $\beta$, $z$ and $M(u_1')$, we deduce $\beta\in [u,v]$.
Now consider the following chain of elements: let $\alpha_0=u_1'$ and, for $k\in [s]$, let $\alpha_k=(b,c_k) \alpha_{k-1}$. Note $\alpha_{k}\rhd \alpha_{k-1}$  and 
\[\alpha_s=[\color{red}ab_1\color{black}c_1\ldots c_s \color{red}b\color{black}]
\]
and, in particular, $\beta=(a,b)\alpha_s$, with $\alpha_s\lhd \beta$. Thus,  $\alpha_k\in [u,v]$, for all $k\in [s]$. We have $M(\alpha_k)=(a,b_1)\alpha_k$, for all $k\in [s]$ (since $a,b_1,b,c_k$ are distinct for all $k$). 

On the other hand, since $\alpha_s= (c_1,b_1)M(z)$ and $\alpha_s\rhd M(z)$, we should have $M(\alpha_s)=(a,b)\alpha_s$, since $a,b,b_1,c_1$ are distinct. This is a contradiction.

We have therefore proved that $b_1$ sits to the right of $c_s$  in $u$ and so there are no elements smaller than $b_1$ and greater than $b$ between $b$ and $c_s$ in $u$. 

By the covering relations in the chain given by the $b_m$, we have that, for all $m$,  there are no digits belonging to $[b_{m-1}, b_{m}]$ between $b$ and $c_s$ in $u$.

Hence, there are no red entries greater than $a$ between $b$ and $c_s$ in $u$, and in particular $c_1,\ldots,c_s$ are all black.
To conclude the proof of the claim, we have to show that, if there is $p$ smaller than $a$ and appearing between $b$ and $c_s$ in $u$,  then $p$ is black. Towards a contradiction, let $p$ be a such red digit. By Proposition~\ref{chains_ij}, we can consider a left basic chain: let $a_1,\ldots, a_m$ with $a_m<p<a_{m-1}\cdots <a_1<a_0=a$ with 
\[
u=[\color{red}a_m\color{black}\ldots \color{red}a_1 a_0 b \color{black}c_1 \ldots c_{k-1} \color{red}p \color{black}c_{k}\ldots  c_s]
\]
for some $k\in [s]$, and 
\[
u_m=[\color{red}a_{m-1}\color{black}\ldots \color{red}a_0 a_m b \color{black}c_1 \ldots c_{k-1} \color{red}p \color{black}c_{k}\ldots  c_s]
\]
with $M(u_m)=(a_m,b) u_m$.
Consider the element 
\[
\alpha=[\color{red}a_{m-1}\color{black}\ldots \color{red}a_0 a_m \color{black}c_1\ldots c_k \color{red}p b \color{black}c_{k+1} \ldots c_s].
\]
Note that $\alpha=(b,c_k)\cdots (b,c_1)u_m$. Note also that the first $m$ rows 
and the last $s-k$ rows of the tableau of $\alpha$ agree with those of $T(M(u_m))$. The remaining rows of $T(\alpha), T(z)$ and $T(M(u_m))$ are shown in the following table:

\begin{center}  
\begin{tabular}{|l|l|l|}
\hline
$T(\alpha)$    &$T(z)$&$T(M(u_m))$  \\
\hline \hline
$a_m\ldots a_0 c_1$& $a_m\ldots a_1 b c_1$& $ a_m\ldots a_0 b$\\
\hline
$a_m\ldots a_0 c_1c_2$& $a_m\ldots a_1 b c_1c_2$& $ a_m\ldots a_0 bc_1$\\
\hline
$\ldots$ &$\ldots$ &$\ldots$ \\
\hline
$a_m\ldots a_0 c_1\ldots c_k$& $a_m\ldots a_1 b c_1\ldots c_k$& $ a_m\ldots a_0 bc_1\ldots c_{k-1}$\\
\hline
$a_mp a_{m-1}\ldots a_0 c_1\ldots c_k$& $a_mp a_{m-1}\ldots a_1 b c_1\ldots c_k$& $ a_m p a_{m-1}\ldots a_0 bc_1\ldots c_{k-1}$\\
\hline
$a_mp a_{m-1}\ldots a_0 b c_1\ldots c_k$&  $a_mp a_{m-1}\ldots a_1 b c_1\ldots c_k c_{k+1}$ & $ a_m p a_{m-1}\ldots a_0 bc_1\ldots c_{k}$\\
\hline

\end{tabular}
\end{center}
where $c_{k+1}=a$ if $k=s$.
It follows that $\alpha \in [u,v]$ and, by Lemma~\ref{absednero}, since $c_1,\ldots,c_k$ are black, we have $M(\alpha)=(a_m,b)\alpha$, which is a contradiction since $a_m<p<b$.

The claim is proved, and hence the result follows.
\end{proof}

\subsection{Coatoms of type (c2$^*$) and (c3$^*$)}

We finally consider the case where $[u,z]$ has simultaneously coatoms of type (c2$^*$) and coatoms of type (c3$^*$). So let $w=(p,q)z$ be a coatom of type (c2$^*$) (so $q$ is a red entry) and $w'=(a,c)z$ be a coatom of type (c3$^*$). If  $z=[\di b \di q \di c \di a \di p\di ]$, then $w'=[\di b\di q\di a\di c\di p]$ and so $M$ would not swap two consecutive red entries in $w$. Hence, $z=[\di b \di c \di q \di a \di p \di ]$ and the letter $c$ must be black otherwise $M$ would not swap two consecutive red entries in $w$. By Lemma~\ref{quali_q}, (\ref{quali_q2*}), there exist $r,s\geq 1$ such that $z$ has the following form:
\[
z=[\color{red} b \color{black}c_1\ldots c_s \color{red}q  a \color{black}p_1\ldots p_r]
\]
with $p_1<\cdots<p_r <q<a<b<c_1<\cdots <c_s$, where all $p_1,\ldots,p_r,c_1,\ldots,c_s$ are black and $q$ is red. Moreover, the coatoms of $[u,z]$ are $M(z)=(a,b)z$, $(p_i,q)z$ for all $i\in [r]$, and $(a,c_j)z$ for all $j\in [s]$.
\begin{lem}
Being $z=[\color{red} b \color{black}c_1\ldots c_s \color{red}q  a \color{black}p_1\ldots p_r]$, we have 
\[u=[\color{red}ab \color{black}c_1\ldots c_{s-1}p_1 c_s p_2 \ldots p_r \color{red}q\color{black}].\] 
\end{lem}
\begin{proof}    
    By the above description of the coatoms, along the interval $[u,z]$ the entries $a,b,c_1,\ldots,c_s$ are permuted among themselves and $q,p_1,\ldots,p_r$ are also permuted among themselves, while all other entries remain in the same positions as in $z$.
    Let $k_1=z^{-1}(b)$. By the previous observation,  $u(k_1)\in \{a,b,c_1,\ldots,c_s\}$, and, since $u<M(z)$, necessarily $u(k)=a$.
    Let $k_2=z^{-1}(c_1)$. We have $u(k_2)\in \{a,b,c_1,\ldots,c_s\}$, and, since $u<(a,c_1)z$, we have $u(k_2)\leq b$ and so $u(k_2)=b$. 
    Next we show that, letting $k_l=z^{-1}(c_{l-1})$ for all $l\in [3,s+1]$, we have $u(k_l)=c_{l-2}$. Indeed,  $u(k_l)\in \{a,b,c_1,\ldots,c_s\}$, and, since $u<(a,c_{l-1})z$, we have $u(k_l)\leq c_{l-1}$, which forces $u(k_l)=c_{l-1}$.
    Now let $k_{s+2}=z^{-1}(q)$. Considering the coatom $(p_1,q)z$, we have $u(k_{s+2})\leq p_1$, and so $u(k_{s+2})=p_1$. For $k_{s+3}=z^{-1}(a)$,  necessarily $u(k_{s+3})=c_s$, since $c_s$ is the only remaining digit of $\{a,b,c_1,\ldots,c_s\}$. It remains to show that, letting $k_l=z^{-1}(p_{l-s-3})$ for $l\in [s+4,s+r+2]$,  we have $u(k_l)=p_{l-s-2}$. This can be obtained with an inductive argument using the coatoms $(p_{l-s-2},q)z$.
\end{proof}
\begin{prop}
\label{prop(c23)}
The element $z$ cannot have coatoms of type \textnormal{(c2$^*$)} and coatoms of type \textnormal{(c3$^*$)}. 
\end{prop}
\begin{proof}
Since $q$ is red, 
by Proposition~\ref{chains_ij} there exists a left basic chain given by $a_1,\ldots,a_m$ with $a_m<q<a_{m-1}<\cdots <a_0=a$ such that
\[
u=[\color{red}a_m\color{black}\ldots \color{red}a_0 b \color{black}c_1\ldots c_{s}p_1 c_s p_2 \ldots p_r \color{red}q\color{black}].
\]
We consider
\[
u_m=[\color{red}a_{m-1}\ldots a_0 a_m b \color{black}c_1\ldots c_{s-1}p_1 c_s p_2 \ldots p_r \color{red}q\color{black}],
\]
wich satifies $M(u_m)=(a_m,b)u_m$, and
\[\alpha=(p_1,q)\cdots(p_r,q)(b,c_s)\cdots (b,c_1)u_m,\] that is,
\[
\alpha=[\color{red}a_{m-1}\color{black}\ldots \color{red}a_0 a_m \color{black}c_1 \cdots c_{s}q\color{red}b \color{black}p_1 \ldots p_r].
\]
As in the previous cases, $\alpha >u_m$. We note that $T(\alpha)$ agrees with $T(M(u_m))$ in the first $m$ rows and with $T(z)$ in the last $r$ rows. The remaining rows are shown in the following table:

\begin{center}  
\begin{tabular}{|l|l|l|}
\hline
$T(\alpha)$    &$T(z)$&$T(M(u_n))$  \\
\hline \hline
$a_m\ldots a_0$& $a_m \ldots a_1 b$& $ a_{m-1}\ldots a_0 b$ \\
\hline 
$a_m\ldots a_0c_1$& $a_m \ldots a_1 b c_1 $& $ a_{m}\ldots a_0 b$ \\
\hline
$\ldots$ & $\ldots$ &$\ldots$ \\
\hline
$a_m\ldots a_0c_1\ldots c_s$& $a_m \ldots a_1 b c_1\ldots c_s $& $ a_{m}\ldots a_0 b c_1 \ldots c_{s-1}$ \\
\hline
$a_mq a_{m-1}\ldots a_0c_1\ldots c_s$& $a_m q a_{m-1}\ldots a_1 b c_1\ldots c_s $& $ p_1 a_{m}\ldots a_0 b c_1 \ldots c_{s-1}$ \\
\hline
$a_mq a_{m-1}\ldots a_0b c_1\ldots c_s$& $a_m q a_{m-1}\ldots a_0 b c_1\ldots c_s $& $ p_1 a_{m}\ldots a_0 b c_1 \ldots c_{s}$ \\
\hline
\end{tabular}
\end{center}

We conclude $\alpha \in [u,v]$. By Lemma~\ref{absednero}, $M(\alpha)=(a_m,b)\alpha$, since $c_1,\ldots,c_s,p_1,\ldots,p_r$ are black. This is a contradiction since $a_m<q<b$.
\end{proof}

\section{Main Results}
\label{mainsection}
In this section, we present the main results of this work, namely a complete classification of special matchings of arbitrary Bruhat intervals in Coxeter groups of type $A$, together with several consequences:
\begin{itemize}
\item we prove that every special matching induces an automorphism of the underlying unoriented Bruhat graph;
\item we prove Brenti's Conjecture;
\item we formulate a conjectural special matchings-based algorithm for computing $R$-polynomials and show that it is equivalent to the Combinatorial Invariance Conjecture.
\end{itemize}

Throughout this section, $(W,S)$ denotes a Coxeter system of type $A$, and $u,v \in W$, with $u\leq v$.

\begin{thm}
\label{main}
Let $M$ be a special matching of $[u,v]$. Then 
\begin{itemize}
    \item there exist $J$ and $s$ such that $$M=\lambda_s^J;$$
    \item there exist $J'$ and $s'$ such that $$M=\rho_{s'}^{J'}.$$
\end{itemize}
\end{thm}
\begin{proof}
Let $J=J(M)$ and $s=s(M)$. By Lemma~\ref{ubase}, we have $M(u)=\lambda_s^J(u)$. If $M\neq \lambda_s^J$ on $[u,v]$ there exists $x\in [u,v]$ such that $M(x)\neq \lambda_s^J(x)$. If this is the case, let $z\in [u,v]$ be minimal in Bruhat order such that $M(z)\lhd z$ and $M(z)\neq \lambda_s^J(z)$. Then $z$ satisfies one of the four mutually exclusive conditions of Proposition \ref{class_finale}. But Propositions~\ref{prop(c1)},~\ref{prop(c2)},~\ref{prop(c3)}, and~\ref{prop(c23)} show that two of these possibilities cannot actually occur. The same propositions applied to the companion matching of $M$ show that also the remaining two conditions are impossible.

The second statement follows by the first statement 
applied to the matching $\overline{M}$ of $[u^{-1},v^{-1}]$ defined by $\overline{M}(x^{-1})= (M(x))^{-1}$.
\end{proof}
\begin{remark}
\label{canonici}
    Given a special matching $M$ of $[u,v]$, there can be more than one  pair $(J,s)$ such that $M=\lambda_s^J$. 
    The canonical choice for $J$ and $s$ is  $J=J(M)$ and $s=s(M)$. Notice that $J(M)$ is also the minimal choice. The canonical choice for $J'$ and $s'$ is $J'=J(\overline M)$ and $s'=s(\overline M)$, where $\overline M$ is the special matching of $[u^{-1},v^{-1}]$ defined in the proof of Theorem \ref{main}.
\end{remark}

Theorem~\ref{main} can be reformulated as follows.
\begin{thm}
\label{prohek}
    Let $M$ be a special matching of $[u,v]$. Then the following hold.
    \begin{itemize}
        \item 
    There exist $i,j,h$ with $1\le h\le j-i$ such that $M$ acts on each $x\in [u,v]$ as follows:
consider the one-line notation of $x$ and the digits $r$ satisfying $i\le r\le j$; among these digits, $M$ swaps the two occupying the $h$-th and $(h+1)$-st positions from left to right.
    
        \item  
        There exist $i',j',h'$ with $1\le h'\le j'-i'$ such that $M$ acts on each $x\in [u,v]$ as follows:
consider the digits occupying positions from $i'$ to $j'$ in the one-line notation of $x$; among these digits, $M$ swaps the $h'$-th smallest and the $(h'+1)$-st smallest.
\end{itemize}
\end{thm}
\begin{remark}
\label{canonici2}
    The canonical choice is $i=i(M)$, $j= j(M)$, $h=h(M)=k-i(M)+1$, where $k$ is such that $s(M)=s_k$, and $i'=i(\overline M)$, $j'=j(\overline M)$, $h'=h(\overline M)$.
\end{remark}

The following consequence of Theorem~\ref{main}
shows that every special matching induces an
automorphism of the underlying unoriented Bruhat graph.

\begin{prop}\label{MdaTaT}
Let $M$ be a special matching of $[u,v]$. Let $x,y\in [u,v]$ and suppose that $y\neq M(x)$ and $x\rightarrow y$ is an arc in the Bruhat graph of $W$. Then also $M(x)\rightarrow M(y)$ is an arc in the Bruhat graph of $W$. 
\end{prop}

\begin{proof}
Let $c,d \in [n]$ be such that $y=(c,d)x$. Let $a,b \in [n]$ be such that $M(x)= (a,b)x$ and let $h,h+1$ be the positions of the digits $a$ and $b$ in the one-line notation of $x$ among red digits. If $M(y)=(a,b)y$, then the assertion follows since $M(y)=(a,b)(c,d)(a,b)M(x)$ and $(a,b)(c,d)(a,b)\in T$ (notice that this transposition is not necessarily $(c,d)$ since $\{c,d\} \cap \{a,b\}$ could be nonempty). We may suppose $M(y)\neq (a,b)y$ and $d=a$.  By Lemma~\ref{absednero}, we have that $c$ is red.
The act of multiplying $x$ by $(c,a)$ put $c$ and $b$ in positions $h$ and $h+1$ in the one line notation of $y$ among red digits. Hence, $M(y)=(c,b)y$. Thus, $M(y)=(c,b)(c,a)(a,b)M(x)$. The assertion is proved since $(c,b)(c,a)(a,b)= (c,a)\in T$.
\end{proof}

 As a first major application of Theorem~\ref{main},
we obtain Brenti's conjecture.

\begin{thm}\label{SMcompute}
Let $W$ be a Coxeter group of type $A$. Let $u,v \in W$, and let $M$ be a special matching of the interval $[u,v]$. Then, for all $x,y$  in $[u,v]$ with $x\leq y$, we have
\[R_{x,y}(q)=\begin{cases}
    R_{M(x),M(y)}(q)& \textrm{if $M(x)\lhd x,\,M(y)\lhd y$ }\\
  R_{M(x),M(y)}(q)  & \textrm{if $M(x)\rhd x,\,M(y)\rhd y$ }\\
    (q-1) R_{x,M(y)}(q)+qR_{M(x),M(y)}(q)& \textrm{if $M(x)\rhd x,\,M(y)\lhd y$ }\\
    q^{-1}R_{M(x),M(y)}(q)+(q^{-1}-1)R_{M(x),y}(q)& \textrm{if $M(x)\lhd x,\,M(y)\rhd y$}
\end{cases}
\]
\end{thm}
\begin{proof}
We proceed by double induction on $\ell(u,v)$ and $\ell(u)$. The case $\ell(u,v)=1$ is trivial and the case $\ell(u)=0$ is the main result of \cite{BCM1}.

We need to show that, 
for all $x,y$ in $[u,v]$ with $x\leq y$ and $M(y)\lhd y$, we have
\[R_{x,y}(q)=\begin{cases}
    R_{M(x),M(y)}(q)& \textrm{if $M(x)\lhd x$}\\
    (q-1) R_{x,M(y)}(q)+qR_{M(x),M(y)}(q)& \textrm{if $M(x)\rhd x$. }
\end{cases}
\]
Indeed, the other two identities follow from these two.

Let $J=J(M)$ and $s=s(M)$. By Theorem~\ref{main} and Remark~\ref{canonici}, we have 
$M(x) = \lambda_s^J(x)$, i.e., $M(x)=\li J x s \lu J x $ for all $x$ in $[u,v]$. By induction hypothesis, we may suppose $y=v$ and $x\in \{u,M(u)\}$.

By Theorem \ref{cha}, we have $\li J x ^K\leq \li J y ^K$, for all $x,y$ in $[u,v]$ with $x\leq y$.

If $\li J v^K=e$, we have $\li J x ^K=e$ for all $x\in [u,v]$, and therefore $M(x)=sx$ and the result is clear. We may suppose $\li J v^K \neq e$ and set $t\in D_L(\li J v^K)$ (observe clearly $t\in J$).

First, suppose $tu\rhd u$. Then, since $t\in D_L(v)$ the matching  $\lambda_t$ restricts to a special matching of $[u,v]$ by Lemma~\ref{si_restringe}. Observe that $s\in D_R((\li J v)_K)$ (since $s\in D_R(\li J v)\cap K$) and therefore $M(v)\neq tv$. Note also that $M$ and $\lambda_t$ commute by construction: indeed, for all $x$ in $[u,v]$, we have $M\lambda_t(x)= t \li J x s  \lu J x$, since $\li J (tx) = t \li J x$.
In particular,  the orbit of $v$ under the action of $\langle M, \lambda_t \rangle$ contains four elements. If $M(u)\neq tu$, then the orbit of $u$ also contains four elements and we have
\begin{align*}
R_{u,v}&=(q-1)R_{u,tv}+qR_{tu, tv}\\
&=(q-1)\big( (q-1)R_{u,M(tv)}+qR_{M(u),M(tv)} \big)+q\big((q-1)R_{tu,M(tv)}+qR_{M(tu),M(tv)}\big)\\
&=(q-1)\big((q-1)R_{u,tM(v)}+qR_{M(u),tM(v)}\big)+q\big((q-1)R_{tu,tM(v)}+qR_{tM(u),tM(v)}\big)\\
&=(q-1)\big((q-1)R_{u,tM(v)}+qR_{tu,tM(v)}\big)+q\big((q-1)R_{M(u),tM(v)}+qR_{tM(u),tM(v)}\big)\\
&= (q-1) R_{u,M(v)}+qR_{M(u),M(v)}
\end{align*}
(notice that, for the fourth equation, we are not claiming that $R_{M(u),tM(v)}$ coincides with $R_{tu,tM(v)}$ but we are just using the fact that these two polynomials appear in the expression with the same coefficient),
and 
similarly 

\begin{align*}
R_{M(u),v}&=R_{u,M(v)}.
\end{align*}

If $M(u)= tu$, then
\begin{align*}
R_{u,v}&=(q-1)R_{u,tv}+qR_{tu, tv}\\
&=(q-1)\big( (q-1)R_{u,M(tv)}+qR_{M(u),M(tv)} \big)+qR_{M(tu),M(tv)}\\
&=(q-1)\big((q-1)R_{u,tM(v)}+qR_{tu,tM(v)}\big)+qR_{u,tM(v)}\\
&=(q-1) R_{u,M(v)} +qR_{tu,M(v)}\\
&=(q-1) R_{u,M(v)} +qR_{M(u),M(v)}
\end{align*}

and similarly 

\begin{align*}
R_{M(u),v}&=R_{u,M(v)}.
\end{align*}

Suppose now $tu\lhd u$. We claim  that $M$ can be extended to a special matching of $[tu,v]$ letting $M(x)=\lambda_s^J( x)$ for all $x\in [tu,v]$. If the claim holds, then we are done since $\ell(tu) < \ell(u)$.

Now we prove the claim. By Theorem~\ref{cha}, we need to show that, for every $x\lhd y$ in $[tu,v]$, we have $\li J x ^K\leq \li J y ^K$, where $K=\{t\in J: ts=st\}$. If $x\in [u,v]$, then $y\in [u,v]$, and $\li J x ^K\leq \li J y ^K$
holds by Theorem~\ref{cha}. Hence, we can suppose $x\notin [u,v]$, and so $tx\rhd x$ by the Lifting Property. If $y=tx$, then $\li J x\lhd t\li J x=\li J y$, and the assertion follows by Lemma~\ref{projection}. So we can suppose $y\neq tx$, and hence $ty\rhd y$ and $ty\rhd tx$. Observe that $tx,ty\in [u,v]$ and so  $\li J (tx) ^K\leq \li J (ty) ^K$. Also note that $\li J (ty)=t \li J y$ and $\li J (tx)=t \li J x$, since $t\in J$. We have three cases to consider.
\begin{enumerate}
    \item If $\li J (ty)=\li J (tx)$, then $\li J x=\li J y$ and the result is clear.
    \item If $\li J (tx)\lhd \li J (ty)$, then $\li J x\lhd \li J y$ and the result also follows by Lemma~\ref{projection}.
    \item If $\li J (tx)>\li J (ty)$,  then  $\li J (tx)^K=\li J (ty)^K$ by Lemma~\ref{projection}, and so $\li J (tx)=\li J (ty)k$, with $k\in W_K$.
    It follows that $\li J x = \li J y k$ and so $\li J x^K=\li J y ^K$.
\end{enumerate}
Hence, the claim is proved.
\end{proof}

Theorem~\ref{SMcompute} suggests a broader
poset-theoretic approach to the computation of
Kazhdan--Lusztig $R$-polynomials.
We present a reformulation of the Combinatorial Invariance Conjecture for Coxeter groups of type $A$. It provides a conjectural algorithm to compute Kazhdan--Lusztig $R$-polynomials in terms of special matchings.

Let  $I$ be a Bruhat interval of rank $r$. We denote by $\mathcal W_{A,r}$ the class of Coxeter groups of rank $r$ that are direct products of Coxeter groups of type $A$.

By classical results of Dyer (see \cite[Proposition~2.1]{DyeComp1}), there exist $\bar{W} \in \mathcal W_{A,r}$ and  $u,v\in \bar{W}$ such that $[u,v]\cong I$. We define the height of $I$ by
\[
ht(I)=\min\{h\geq 0:\, \textrm{ there exist $\bar{W}\in \mathcal W_{A,r}$ and $u,v\in \bar{W}$ with $\ell(u)=h$ and $[u,v]\cong I$}\}.
\]
It was proved in \cite{BG} (and it is also a consequence of Theorem~\ref{SMcompute}) that the Combinatorial Invariance Conjecture holds for intervals of height 0.

Recall that a filter in a poset $I$ is a subset $F\subseteq I$ such that for all $x,y\in I$, if $x\in F$ and $x\leq y$ then $y\in F$.

\begin{defn}
A partial special matching of $[u,v]$ is a special matching of a non-empty filter of $[u,v]$.\end{defn}

\begin{remark}
Partial special matchings should not be confused with special partial matchings studied in \cite{AHH,AH,CMpir,Mpir}.
\end{remark}

Given a partial special matching $M$ of an interval $I$ defined on a filter $F$, one can define an extension $I^M$ of $I$ by adding a duplicate of rank $r-1$ for every element of rank $r$ in $I\setminus F$ (see Figure \ref{fig:exte} for an example). We extend the partial special matching $M$ to a complete matching of the set $I^M$ (which we still denote by $M$), by matching every element in $I\setminus F$ with its duplicate. Covering relations on $I^M$ are uniquely determined by the requirement that $M$ be a special matching of $I^M$. Namely, we add the following covering relations:
for all $x \in I\setminus F$, we add $M(x)\lhd y$ if and only if either $y=x$, or $M(y)\rhd x$ and $M(y)\rhd y$.

\begin{remark}
By construction, the matching $M$ on $I^M$ is a (complete) special matching.
\end{remark}

\begin{figure}[h] 
\centering
    \begin{tikzpicture}[scale=0.80]
      \draw (0,0) node[fill=white, thick, draw=black, inner sep=0,outer sep=0.5mm, minimum size=2mm, circle](u1) {};
      \draw (-3,2) node[fill=white, thick, draw=black, inner sep=0,outer sep=0.5mm, minimum size=2mm, circle](u2) {};
\draw (-1,2) node[fill=white, thick, draw=black, inner sep=0,outer sep=0.5mm, minimum size=2mm, circle](u3) {};
\draw (1,2) node[fill=white, thick, draw=black, inner sep=0,outer sep=0.5mm, minimum size=2mm, circle](u4) {};
\draw (3,2) node[fill=white, thick, draw=black, inner sep=0,outer sep=0.5mm, minimum size=2mm, circle](u5) {};
\draw (-3,4) node[fill=white, thick, draw=black, inner sep=0,outer sep=0.5mm, minimum size=2mm, circle](u6) {};
\draw (-1,4) node[fill=white, thick, draw=black, inner sep=0,outer sep=0.5mm, minimum size=2mm, circle](u7) {};
\draw (1,4) node[fill=white, thick, draw=black, inner sep=0,outer sep=0.5mm, minimum size=2mm, circle](u8) {};
\draw (3,4) node[fill=white, thick, draw=black, inner sep=0,outer sep=0.5mm, minimum size=2mm, circle](u9) {};
\draw (0,6) node[fill=white, thick, draw=black, inner sep=0,outer sep=0.5mm, minimum size=2mm, circle](u10) {};
\draw (10,0) node[fill=white, thick, draw=black, inner sep=0,outer sep=0.5mm, minimum size=2mm, circle](v1) {};
\draw (7,2) node[fill=white, thick, draw=black, inner sep=0,outer sep=0.5mm, minimum size=2mm, circle](v2) {};
\draw (9,2) node[fill=white, thick, draw=black, inner sep=0,outer sep=0.5mm, minimum size=2mm, circle](v3) {};
\draw (11,2) node[fill=white, thick, draw=black, inner sep=0,outer sep=0.5mm, minimum size=2mm, circle](v4) {};
\draw (13,2) node[fill=white, thick, draw=black, inner sep=0,outer sep=0.5mm, minimum size=2mm, circle](v5) {};
\draw (7,4) node[fill=white, thick, draw=black, inner sep=0,outer sep=0.5mm, minimum size=2mm, circle](v6) {};
\draw (9,4) node[fill=white, thick, draw=black, inner sep=0,outer sep=0.5mm, minimum size=2mm, circle](v7) {};
\draw (11,4) node[fill=white, thick, draw=black, inner sep=0,outer sep=0.5mm, minimum size=2mm, circle](v8) {};
\draw (13,4) node[fill=white, thick, draw=black, inner sep=0,outer sep=0.5mm, minimum size=2mm, circle](v9) {};
\draw (10,6) node[fill=white, thick, draw=black, inner sep=0,outer sep=0.5mm, minimum size=2mm, circle](v10) {};
\draw (12,0) node[fill=white, thick, draw=black, inner sep=0,outer sep=0.5mm, minimum size=2mm, circle](v11) {};
\draw (14,0) node[fill=white, thick, draw=black, inner sep=0,outer sep=0.5mm, minimum size=2mm, circle](v12) {};
\draw (12,-2) node[fill=white, thick, draw=black, inner sep=0,outer sep=0.5mm, minimum size=2mm, circle](v13) {};
\draw (15,2) node[fill=white, thick, draw=black, inner sep=0,outer sep=0.5mm, minimum size=2mm, circle](v14) {};

\draw[thick] (u1)--(u2);
\draw[thick] (u1)--(u3);
\draw[thick] (u1)--(u4);
\draw[thick] (u1)--(u5);
\draw[thick] (u2)--(u6);
\draw[thick] (u3)--(u6);
\draw[thick] (u4)--(u7);
\draw[thick] (u4)--(u9);
\draw[thick] (u5)--(u8);
\draw[thick] (u5)--(u9);
\draw[thick] (u7)--(u10);
\draw[thick] (u8)--(u10);
\draw[thick] (u9)--(u10);
\draw[line width=1mm, color=red] (u2)--(u7);
\draw[line width=1mm, color=red] (u3)--(u8);
\draw[line width=1mm, color=red] (u6)--(u10);

\draw[thick] (v1)--(v2);
\draw[thick] (v1)--(v3);
\draw[thick] (v1)--(v4);
\draw[thick] (v1)--(v5);
\draw[thick] (v2)--(v6);
\draw[thick] (v3)--(v6);
\draw[thick] (v4)--(v7);
\draw[thick] (v4)--(v9);
\draw[thick] (v5)--(v8);
\draw[thick] (v5)--(v9);
\draw[thick] (v7)--(v10);
\draw[thick] (v8)--(v10);
\draw[thick] (v9)--(v10);
\draw[line width=1mm, color=red] (v2)--(v7);
\draw[line width=1mm, color=red] (v3)--(v8);
\draw[line width=1mm, color=red] (v6)--(v10);
\draw[thick] (v11)--(v2);
\draw[thick] (v12)--(v3);
\draw[thick] (v11)--(v14);
\draw[thick] (v12)--(v14);
\draw[thick] (v13)--(v11);
\draw[thick] (v13)--(v12);
\draw[thick] (v14)--(v6);
\draw[line width=1mm, color=red] (v13)--(v1);
\draw[line width=1mm, color=red] (v11)--(v4);
\draw[line width=1mm, color=red] (v12)--(v5);
\draw[line width=1mm, color=red] (v14)--(v9);

\end{tikzpicture}
    \caption{A partial special matching (left) and its extension (right)}
    \label{fig:exte}
\end{figure}
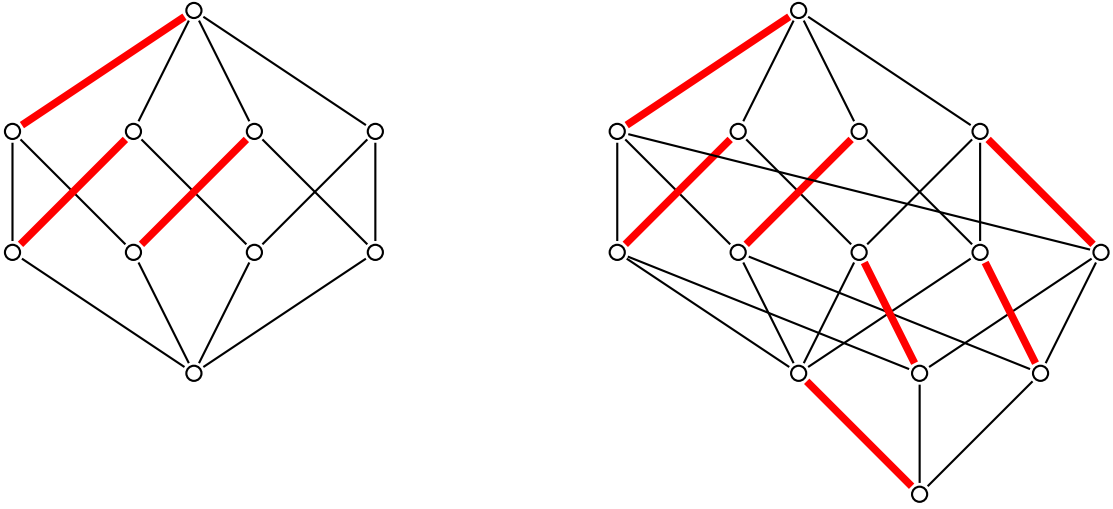
\begin{lem}\label{lempartial}
Let $I=[u,v]$ be a Bruhat interval of rank $r$ and height $h$ with no special matchings. Then there exists a partial special matching $M$ of $I$ such that $I^M$ is isomorphic to a Bruhat interval $[u',v']\subseteq  W(A_r)$ with $\ell(u')=h-1$. In particular, the subinterval $[M(u),M(v)]$ of $I^M$ is isomorphic to a Bruhat interval of height at most $h-1$.
\end{lem}
\begin{proof}We can clearly assume that $u,v\in W(A_r)$ and $\ell(u)=h$. 
Let $s\in D_L(v)$ and let $F=\{x\in [u,v]:\, sx \in [u,v]\}$. Note that $sv\in [u,v]$ otherwise $[su,sv]\cong[u,v]$ contradicting the condition $ht([u,v])=h$. So $v,sv\in F$. Note that $F$ is a filter in $[u,v]$ by the Lifting Property and, since $[u,v]$ has no special matchings, we have $F \subsetneq [u,v]$. We consider the partial special matching $M$  on $[u,v]$ given by the left multiplication by $s$ on $F$. Then $I^M$ is isomorphic to $[su,v]$ and $[M(u),M(v)]$ is isomorphic to $[su,sv]$. The result follows.
\end{proof}

Theorem~\ref{SMcompute} and Lemma~\ref{lempartial} lead us to consider the following conjectural algorithm for computing $R$-polynomials.

\begin{algo}
    Let $I=[u,v]$ be a Bruhat interval of rank $r$ and height $h$. 
\begin{itemize}
\item If $I$ has a special matching $M$, then
\[
R_{u,v}=(q-1)R_{u,M(v)}+q R_{M(u),M(v)}.
\]
\item If $I$ has no special matchings, let $M$ be a partial special matching of $[u,v]$, $u',v' \in W(A_r)$ with $\ell(u')=h-1$, and $\varphi:I^M \to [u',v']$ be an isomorphism. 
Then
\[
R_{u,v}=R_{\varphi(M(u)),\varphi (M(v))}.
\]
\end{itemize}
\end{algo}

\begin{remark} 
The algorithm terminates since, at each step, either the rank decreases or it remains the same but the height decreases. The algorithm is also effective since it involves a finite number of checks.
\end{remark}

\begin{conj}\label{newconj}
The algorithm computes the $R$-polynomials.
\end{conj}

\begin{remark} 
Implicit in Conjecture \ref{newconj} lies the fact that the algorithm does not depend on the choices of the partial special matchings.
\end{remark}

\begin{prop}
 Conjecture~\ref{newconj} is equivalent to the Combinatorial Invariance Conjecture for Coxeter groups of type $A$.
\end{prop}
\begin{proof}
Since the algorithm depends only on the isomorphism class of the interval,  Conjecture~\ref{newconj} implies the Combinatorial Invariance Conjecture. 

Suppose that the Combinatorial Invariance Conjecture holds. If the interval $I$ has a special matching, then the result follows by  Theorem~\ref{SMcompute}. If $I$ has no special matchings, let $M$  and $\varphi$ as in the statement of the algorithm. 
Then \[
R_{u,v} = R_{\varphi (u),\varphi (v)}=R_{\varphi(M) (\varphi (u)),\varphi(M)( \varphi (v))} = R_{\varphi (M(u)),\varphi (M(v))}  
\]
where the first equality follows from the Combinatorial Invariance Conjecture, the second equality by  Theorem~\ref{SMcompute} applied to the special matching $\varphi(M)$ of $\varphi(I^M)$ defined by $\varphi(M)(\varphi(x))=\varphi(M(x))$ for every $x\in I^M$.
\end{proof}

\bigskip
\section*{Acknowledgments}
This project benefited from stimulating discussions during the scientific meeting
\lq\lq Algebraic Combinatorics in Ancona\rq\rq,
held at DIISM, Universit\`a Politecnica delle Marche, in Ancona, in June 2025.
M.M. is a member of the INDAM research group GNSAGA.
F.C. acknowledges support from PRIN 2022S8SSW2,
\lq\lq Algebraic and Geometric Aspects of Lie Theory\rq\rq.
M.M. acknowledges support from PRIN 2022A7L229,
\lq\lq ALgebraic and TOPological Combinatorics\rq\rq.
F.C. would also like to thank Universit\`a Politecnica delle Marche for its hospitality.


\begin{thebibliography}{xx}

\bibitem{AHH}
N. Abdallah, M. Hansson, A. Hultman, {\em Topology of posets with special partial matchings}, Adv. Math. {\bf 348} (2019), 255–276.

\bibitem{AH}
N. Abdallah, A. Hultman, {\em Combinatorial invariance of Kazhdan–Lusztig–Vogan polynomials for fixed point free involutions}, J. Algebraic Combin. {\bf 47} (2018), 543–560.

\bibitem{BG}
G.T. Barkley, C.  Gaetz, {\em Combinatorial invariance for elementary intervals}, Math. Ann. {\bf 392} (2025), 3299–3317.

\bibitem{BGL} G.T. Barkley, C. Gaetz, T. Lam, {\em Combinatorial invariance for the coefficient of $q$ in Kazhdan-Lusztig polynomials}, arXiv:2601.07793 [math.CO].

\bibitem{BB}
A. Bj\"{o}rner, F. Brenti, {\em Combinatorics of Coxeter Groups}, Graduate Texts in Mathematics, {\bf 231}, Springer-Verlag, New York, 2005.

\bibitem{BBDVW}
 C. Blundell,  L. Buesing, A. Davies,  P. Veli\u ckovi\'c, G. Williamson, {\em  Towards combinatorial invariance for Kazhdan-Lusztig polynomials}, Represent. Theory {\bf 26} (2022), 1145-1191.

\bibitem{Bre94}
F. Brenti, 
{\em A combinatorial formula for Kazhdan-Lusztig polynomials}, 
Invent. Math. {\bf 118} (1994), 371-394.

\bibitem{Bre03}
F. Brenti, 
{\em Kazhdan--Lusztig polynomials: History,
Problems, and Combinatorial Invariance}, 
S\'eminaire Lotharingien de Combinatoire {\bf 49} (2003), Article B49bJ.

\bibitem{Bre04}
F. Brenti, 
{\em The intersection cohomology of Schubert varieties
is a combinatorial invariant}, 
J. European Combin. {\bf 25} (2004), 1151-1167.

\bibitem{BCM1}
F. Brenti, F. Caselli, M. Marietti, {\em Special Matchings and Kazhdan--Lusztig polynomials},  Adv. Math. {\bf 202} (2006), 555-601.

\bibitem{BCM2}
F. Brenti, F. Caselli, M. Marietti, {\em Diamonds and Hecke algebra representations}, Int. Math. Res. Not. IMRN, {\bf 2006} (2006), 29407.

\bibitem{B-M}
F. Brenti, M. Marietti, {\em Kazhdan--Lusztig R-polynomials, combinatorial invariance, and hypercube decompositions}, Math. Z. {\bf 309} 25 (2025).

\bibitem{BLP}
G. Burrull, N. Libedinsky, D. Plaza, 
{\em Combinatorial invariance conjecture for $\tilde{A}_2$}, Int. Math. Res. Not. IMRN, {\bf 2023}, Issue 10, (2023), 8903–8933.

\bibitem{CDM}
F. Caselli, M. D'Adderio, M. Marietti, {\em Weak Generalized Lifting Property, Bruhat Intervals, and Coxeter Matroids},  Int. Math. Res. Not. IMRN, {\bf 2021}, Issue 3, (2021), 1678–1698.

\bibitem{CM1}
F. Caselli, M. Marietti, {\em Special matchings and Coxeter groups},
Europ. J. Combin., {\bf 61} (2017), 151--166.

\bibitem{CM2}
F. Caselli, M. Marietti, {\em A simple characterization of special matchings in lower Bruhat intervals},
Discrete Math., {\bf 341} (2018), 851--862.

\bibitem{CMpir}
F. Caselli, M. Marietti, {\em Pircon kernels and up-down symmetry}, J. Algebra {\bf 565} (2021), 324–352.

\bibitem{DVBBZTTBBJLWHK}
A. Davies, P. Veli\u ckovi\'c, L. Buesing, S. Blackwell, D. Zheng, N. Tomašev, R. Tanburn,  P. Battaglia, C. Blundell, A. Juhász, M. Lackenby, G. Williamson, D. Hassabis,  P. Kohli, {\em Advancing mathematics by guiding human intuition with AI}, Nature, {\bf 600} (2021), 70-74.

\bibitem{Dyeth}
M. J. Dyer, {\em Hecke algebras and reflections in Coxeter groups},  Ph. D. Thesis, University of Sydney, 1987.

\bibitem{DyeComp1}
M. J. Dyer, {\em On the \lq\lq Bruhat graph\rq\rq of a Coxeter system}, Compos. Math., {\bf 78} (1991),  185-191. 

\bibitem{EM}
F. Esposito, M. Marietti, {\em Flipclasses and Combinatorial Invariance for Kazhdan--Lusztig polynomials},  Sel. Math. New Ser. 31, 98 (2025). https://doi.org/10.1007/s00029-025-01099-6

\bibitem{EM2}
F. Esposito, M. Marietti, {\em A note on Combinatorial Invariance of Kazhdan--Lusztig polynomials (with an appendix by  G. T. Barkley and C. Gaetz)}, preprint  arXiv:2404.12834v3.

\bibitem{EMS}
F. Esposito, M. Marietti, Stella, {\em Flip Combinatorial Invariance and Weyl Groups}, preprint arXiv:2509.16433 [math.CO].

\bibitem{G-W}
M. Gurevich, C. Wang, {\em Parabolic recursions for Kazhdan-Lusztig polynomials and the hypercube decomposition}, Sel. Math. New Ser. 30, 81 (2024). https://doi.org/10.1007/s00029-024-00972-0

\bibitem{K-L}
D. Kazhdan, G. Lusztig, {\em Representations of Coxeter groups and Hecke algebras}, Invent. Math. {\bf 53} (1979), 165-184.

\bibitem{MJaco} M. Marietti, {\em Algebraic and combinatorial properties of zircons}, J. Algebraic Combin., {\bf 26} (2007), 363-382.

\bibitem{Mtrans} M. Marietti, {\em Special matchings and parabolic Kazhdan--Lusztig polynomials},  Trans. Amer. Math. Soc. {\bf 368} (2016), no. 7, 5247-5269.

\bibitem{M} M. Marietti, {\em The combinatorial invariance conjecture for parabolic Kazhdan--Lusztig polynomials of lower intervals},  Advances in Math. {\bf 335} (2018), 180-210.

\bibitem{Mpir}
M. Marietti, {\em Kazhdan–Lusztig R-polynomials for pircons}, J. Algebra {\bf 534} (2019), 245–272.

\end{thebibliography}
\end{document}